\documentclass[reqno]{amsart}
\usepackage{hyperref}
\usepackage{amssymb}
\usepackage{amsthm}
\usepackage{amsmath}

\newtheorem{theorem}{Theorem}[section]
\newtheorem{lemma}[theorem]{Lemma}

\newtheorem{proposition}[theorem]{Proposition}

\theoremstyle{definition}
\newtheorem{definition}[theorem]{Definition}
\newtheorem{example}[theorem]{Example}

\theoremstyle{remark}
\newtheorem{remark}[theorem]{Remark}
\allowdisplaybreaks
\numberwithin{equation}{section}

\begin{document}

\title[$\varphi $-strong solutions and uniqueness of SDE\lowercase{s}]{$\varphi $-strong solutions and uniqueness of $1$-dimensional stochastic
differential equations}

\author{Mihai N. Pascu}
\address{Faculty of Mathematics and Computer Science, Transilvania University of Bra\c{s}ov, Str. Iuliu Maniu Nr. 50, Bra\c{s}ov -- 500091, ROMANIA}
\email{mihai.pascu@unitbv.ro}

\begin{abstract}
In this paper we consider stochastic differential equations with
discontinuous diffusion coefficient of varying sign, for which weak existence and uniqueness holds but strong
uniqueness fails. We introduce the notion of $\varphi $-strong solution, and
we show that under certain conditions on the diffusion coefficient a $\varphi
$-strong solution exists and it is unique. We also give a construction of a $\varphi$-strong solution and a weak solution in terms of the notion of \emph{i.i.d. sign choice} introduced in the paper, and we give an explicit representation of the set of all weak solutions, which in particular explains the reason for the lack of strong existence and uniqueness. 
\end{abstract}
\maketitle

\section{Introduction}

This paper aims to give a better understanding of the gap between weak and strong existence and uniqueness of solutions of $1$-dimensional stochastic differential equations (SDEs), more precisely the case when weak existence and uniqueness in the sense of probability law holds but strong existence or strong uniqueness fails. To this end,  we introduce the notion of $\varphi$-strong solution (and the corresponding notions of existence and uniqueness), which interpolates between a strong solution and a weak solution of a SDE. In Theorem \ref{theorem 1} and Theorem \ref{theorem 2} we study two classes of SDEs with discontinuous diffusion coefficient of varying sign, for which we prove $\varphi$-strong existence and uniqueness for a suitable function $\varphi$. We also give an explicit construction of a $\varphi$-strong solution (also a weak solution) in terms of an \emph{i.i.d. sign choice} for a certain process, and we give a representation of all weak solutions. This representation explains, on one hand, the reason for uniqueness in the sense of probability law of the weak solutions, and on the other it explains the lack of strong existence and uniqueness: a generic weak solution is not determined only by the driving Brownian motion of the SDE, but also by a process independent of it (the i.i.d. sign choice).

Consider the stochastic differential equation%
\begin{equation}
X_{t}=X_{0}+\int_{0}^{t}\sigma \left( X_{s}\right)
dB_{s}+\int_{0}^{t}b\left( X_{s}\right) ds,\qquad t\geq 0,  \label{SDE}
\end{equation}
where $\sigma$ and $b$ are Borel-measurable functions and $B$ is a $1$-dimensional Brownian motion on a fixed probability space $(\Omega,\mathcal{F},P)$.



When referring to strong/weak solutions and existence/uniqueness of the
above SDE, we will use the classical terminology (see e.g., \cite%
{Karatzas-Shreve}).

The problem of existence of weak solutions and uniqueness in the sense of proba\-bility law for (\ref{SDE}) was completely solved by Engelbert and Schmidt (\cite{Engelbert-Schmidt '85}) in the case when the drift is identically zero. They showed that for any initial distribution of $X_0$, there exists a weak solution if and only if $I(\sigma)\subset Z(\sigma)$, and the solution is unique in the sense of probability law if and only if  $I(\sigma)=Z(\sigma)$, where $I(\sigma)=\{x\in \mathbb{R}: \int_{x-\varepsilon}^{x+\varepsilon} \frac{1}{\sigma^2(y)} dy =\infty,\quad \forall \varepsilon >0\}$ is the set of non-local integrability of $\sigma^{-2}$, and $Z(\sigma)=\{x\in \mathbb{R}: \sigma(x)=0\}$ is the set of zeroes of $\sigma$. Using the method of removal of drift, the case of a general drift can be reduced to the case when the drift is identically zero, thus obtaining sufficient conditions for weak existence and uniqueness (see \cite{Karatzas-Shreve}, pp. 339 -- 342).

One of the first results regarding the strong solutions of (\ref{SDE}) is due to the pioneering work of K. It\^{o} (\cite{Ito '46}), who showed that if the drift and the diffusion coefficient are Lipschitz  continuous functions, then strong existence and uniqueness holds for (\ref{SDE}). There are other sufficient conditions for the strong existence and uniqueness in the literature; for example, it is known that if $\sigma $ is Lipschitz continu\-ous of exponent $\frac{1}{2}$ and $b$ is Lipschitz continuous (\cite{Yamada-Watanabe}), or if $\sigma$, $b$ are bounded measurable functions, $\sigma $ satisfies $\vert \sigma(x)-\sigma(y)\vert \le \rho(\vert x-y \vert )$, $x,y\in\mathbb{R}$, for an increasing function $\rho$ with $\int_{0+}\frac{1}{\rho^2(x)} dx =\infty$, and $\vert \sigma \vert$ is bounded below away from zero (\cite{Le Gall}), there exists a unique strong solution of (\ref{SDE}) above. In particular, these hypotheses require the continuity of the diffusion coefficient $\sigma $. Extending a result of Nakao, Le Gall (\cite{Le Gall}) showed that if $\sigma, b$ are bounded measurable functions, $\sigma $ is bounded below away from zero and it satisfies the condition
\begin{equation}\label{Nakao condition}
\left\vert \sigma \left( x\right) -\sigma \left( y\right) \right\vert
^{2}\leq \left\vert f\left( x\right) -f\left( y\right) \right\vert ,\qquad
x,y\in \mathbb{R},  
\end{equation}%
for some bounded increasing function $f$, there exists a unique strong solution of (\ref{SDE}).

Note that the hypothesis (\ref{Nakao condition}) allows for countable many discontinuities of $\sigma $, but as a drawback, $\sigma $ (instead of $%
\left\vert \sigma \right\vert $) is required to be bounded below away from
zero. As pointed out in \cite{Le Gall}, these conditions are almost sharp,
since by relaxing any of the two conditions on $\sigma $ one can produce
SDEs for which the uniqueness fails: Barlow (\cite{Barlow}) showed that for
any $\alpha >2$ one can construct functions $\sigma $ of bounded variation
of order $\alpha $ which are bounded below away from zero, such that there
is not pathwise uniqueness for the SDE $X_{t}=\int_{0}^{t}\sigma \left(
X_{s}\right) dB_{s}$. Also, the classical example (attributed to Tanaka)
when $\sigma \left( x\right) =\mathrm{sgn}\left( x\right) $ and $b\equiv 0$
shows that we cannot replace the condition that $\sigma $ is bounded below
away from zero by the condition that $\left\vert \sigma \right\vert $ is
bounded below away from zero.

In this paper, we consider the case of SDEs for which the diffusion
coefficient $\sigma $ satisfies Le Gall's condition (\ref{Nakao condition})
and $\left\vert \sigma \right\vert $ is bounded below away from zero. As a
first step, we consider the SDE (\ref{SDE}) in the case when $b\equiv 0$, $%
\sigma $ has only one jump discontinuity at the origin, $\sigma$ is an odd function on $\mathbb{R}^\ast=\mathbb{R}-\{0\}$ (or a step function), and satisfies $x\sigma
\left( x\right) \geq 0$ for all $x\in \mathbb{R}$. Although under these
conditions strong existence and uniqueness may fail, we show that it is
possible to describe explicitly the set of all weak solutions. The generic
weak solution depends on an \emph{i.i.d. sign choice} (in the sense of
Definition \ref{iid sign choice}), and it is uniquely determined given such a
sign choice. This shows that although we cannot determine uniquely the
``output'' $X_t$ for a given the ``input'' $\left(B_s\right)_{0\le s\le t}$ from (\ref{SDE}), it may
be possible to determine uniquely a certain function $\varphi \left(X_t\right)$ of it.

This justifies the introduction of the notion of $\varphi $\emph{-strong solution }of an SDE (see Definition \ref{varphi strong solution}), which interpolates between the classical notions of weak solution and strong solution of an SDE.

We were led to considering the above type of SDEs by our recent work on
extending the mirror coupling of reflecting Brownian motions to the case
when the two processes live in different domains. In the simplest $1$%
-dimensional case, the problem can be reduced to constructing a solution of
the singular SDE%
\begin{equation}
X_{t}=\int_{0}^{t}\mathrm{sgn}\left( X_{s}-B_{s}\right) dB_{s},\qquad t\geq 0,
\end{equation}%
corresponding to a diffusion coefficient $\sigma $ of varying sign, with a
discontinuity at the origin, for which $\left\vert \sigma \right\vert =1$ is
bounded below away from zero.

The solution of the above SDE is not strongly unique, but one can construct
a carefully chosen weak solution $X$ with the property that $X$ and $B$
spend the least amount of time together (other choices are also possible).
This gives rise to the \emph{mirror coupling} of reflecting Brownian motions, and as
an application of this construction in  \cite{Pascu} we obtained a unifying proof of I. Chavel
and W. S. Kendall's results on Chavel's conjecture on the domain
monotonicity of the Neumann heat kernel.

The structure of the paper is as follows. In Section \ref{Section 1} we
first consider the classical example, attributed to Tanaka, when the diffusion
coefficient is given by $\sigma \left( x\right) =\mathrm{sgn}\left(
x\right) $, and we show that although the solution is not pathwise unique, its
absolute value is pathwise unique. Next, we introduce the
notions of \emph{$\varphi $-strong solution} of an SDE, \emph{sign choice} and \emph{i.i.d. sign choice} for a
process.

In Theorem \ref{theorem 1} we consider the case when the diffusion
coefficient is a step
function $\sigma _{a,b}=a1_{[0,\infty )}+b1_{(-\infty ,0)}$ for constants $a>0>b$, and we show that $\varphi _{a,b}$-strong
uniqueness holds in this case, with $\varphi _{a,b}\left( x\right)
=\int_{0}^{x}\frac{1}{\sigma_{a,b} \left( u\right) }du$. Moreover, we show that
any weak solution has the representation $X_t=\sigma _{a,b}\left( U_{t}\right)
Y_{t}$, $t\ge 0$, where $Y$ is the reflecting Brownian motion on $[0,\infty)$
with driving Brownian motion $B$ and $U$ is an i.i.d. sign choice for $Y$ (in the sense of Definition \ref{iid sign choice}).

In Section \ref{Section 2} we consider the case of an SDE for which the
diffusion coefficient $\sigma $ is an odd function on $\mathbb{R}^\ast$ satisfying Le Gall's
condition (\ref{Nakao condition}), $x\sigma(x)\geq 0$ for $x\in\mathbb{R}$, and $\left\vert \sigma \right\vert $ is
bounded between positive constants. In Theorem \ref{theorem 2} we show that $\left\vert x\right\vert $-strong uniqueness holds in this case, and
that any weak solution has the representation $X_{t}=U_{t}Y_{t}$, $t\ge 0$, where $Y$ is a certain time changed reflecting Brownian motion on $[0,\infty)$ and $U$ is an i.i.d. sign choice for $Y$.

We conclude with some remarks on the possibility of extending the theorem to
the general case of a diffusion coefficient $\sigma $ satisfying Le Gall's
condition (\ref{Nakao condition}) and for which $\left\vert \sigma \right\vert $ is bounded between
positive constants.

\section{A particular case\label{Section 1}}

Consider the particular case of (\ref{SDE}) when $\sigma \left( x\right) =%
\mathrm{sgn}(x)=\left\{
\begin{tabular}{ll}
$+1,$ & $x\geq 0$ \\
$-1,$ & $x<0$%
\end{tabular}%
\right. $ and $b\left( x\right) \equiv 0$, that is%
\begin{equation}
X_{t}=\int_{0}^{t}\mathrm{sgn}\left( X_{s}\right) dB_{s},\qquad t\geq 0.
\label{SDE1}
\end{equation}

This is a classical example (first considered by Tanaka, and later by many
authors, e.g. Zvonkin \cite{Zvonkin}, or \cite{Karatzas-Shreve})
when the strong uniqueness fails. Note that since $\mathrm{sgn}(x)$ is a
discontinuous function, the classical results on strong existence and
uniqueness of solutions (see for example \cite{Le Gall}, or \cite{Karatzas-Shreve}, p. 291) do not apply here. Also,
since $\mathrm{sgn}^{-2}(x)$ is everywhere integrable and $\mathrm{sgn}(x)$ has no zeroes, by the
results of Engelbert and Schmidt it follows that the SDE (\ref{SDE1}) has a weak solution and the
solution is unique in the sense of probability law.

\begin{remark}
It is not difficult to see directly that this is indeed the case.

Following \cite{Karatzas-Shreve} (Section 5.3, pp. 301 -- 302), first note that if $\left( X,B,\mathbb{F}\right) $ is a weak solution for (\ref{SDE1}), then $X$ is a
continuous, square integrable martingale with quadratic variation $\langle
X\rangle _{t}=t$, so by L\'{e}vy's theorem $X$ is a Brownian motion, and
therefore uniqueness in the sense of probability law holds for (\ref{SDE1}).

To see that (\ref{SDE1}) has a weak solution, one can consider a $1$-dimensional
Brownian motion $\left( X, \mathbb{F}\right)$, and define $B_{t}=\int_{0}^{t}\mathrm{sgn}\left(
X_{s}\right) dX_{s}$, $t\ge 0$, so $\left(B, \mathbb{F}\right)$ is a $1$-dimensional Brownian motion. Then $%
\int_{0}^{t}\mathrm{sgn}\left( X_{s}\right) dB_{s}=\int_{0}^{t}\left(
\mathrm{sgn}\left( X_{s}\right) \right) ^{2}dX_{s}=X_{t}$, so $\left(
X,B,\mathbb{F}\right)$ is a weak solution of (\ref{SDE1}). As a consequence, since $X$ and $-X$ are solutions of (\ref{SDE1}) at the same time,
pathwise uniqueness cannot hold for (\ref{SDE1}).

If (\ref{SDE1}) had a strong solution $X$ with respect to a fixed Brownian motion $B$ (so $\mathcal{F}_{t}^{X}\subset \mathcal{F}%
_{t}^{B}$ for all $t\geq 0$, where $\left(\mathcal{F}_t^X\right)_{t\ge 0}$ and $\left(\mathcal{F}_t^B\right)_{t \ge 0}$ denote the augmented filtrations generated by $X$, respectively $B$), then $B_{t}=\int_{0}^{t}\mathrm{sgn}\left(
X_{s}\right) dX_{s}$, so $\mathcal{F}_{t}^{B}\subset \mathcal{F}_{t}^{X}$,
and therefore $\mathcal{F}_{t}^{X}=\mathcal{F}_{t}^{B}$.

By the Tanaka formula (although our definition $\mathrm{sgn}(0)=1$ differs from the usual left continuous one, the calculations are the same since Brownian motion spends zero time at the origin) we also have $\int_{0}^{t}\mathrm{sgn}\left( X_{s}\right)
dX_{s}=\left\vert X_{t}\right\vert -\widehat{L}_{t}^{0}(\vert X\vert)$, where $\widehat{L}_t^0(\vert X\vert)=\lim_{%
\varepsilon \searrow 0} \frac{1}{2\varepsilon} \int_0^t 1_{[0, \varepsilon)}(\vert X_s\vert) ds$ is the symmetric local time
of $X$ at the origin, so $\mathcal{F}_{t}^{B}\subset \mathcal{F}%
_{t}^{\left\vert X\right\vert }$, which leads to a contradiction: $\mathcal{F%
}_{t}^{X}\subset \mathcal{F}_{t}^{\left\vert X\right\vert }$. The
contradiction shows that (\ref{SDE1}) does not have a strong solution.
\end{remark}

A stochastic differential equation like (\ref{SDE}) or (\ref{SDE1}) above
can be thought as a machinery which given the ``input'' $\left(B_{s}\right)_{0\le s\le t}$ produces the ``output'' $X_{t}$ (the \emph{principle of causality} for
dynamical systems, according to \cite{Karatzas-Shreve}). The fact that (\ref%
{SDE1}) does not have a strong solution shows that $X_{t}$ cannot be
determined from the input $\left(B_{s}\right)_{0\le s \le t}$. Even though the output $X_{t}$ cannot be
`` predicted'' from the input $\left(B_{s}\right)_{0\le s \le t}$, we can still say something about a process $X$ which satisfies (\ref{SDE1}).

Using again Tanaka formula as above, if $X$ verifies (\ref{SDE1}) we obtain
\[
\left\vert X_{t}\right\vert =\int_{0}^{t}\mathrm{sgn}\left( X_{s}\right)
dX_{s}+\widehat{L}_{t}^{0}(\vert X\vert)=\int_{0}^{t} \mathrm{sgn}^2\left( X_{s}\right) dB_{s}+\widehat{L}_{t}^{0}(\vert X \vert)=B_{t}+\widehat{L}_{t}^{0}(\vert X \vert ),
\]
where $\widehat{L}_t^0(\vert X\vert)$ is the symmetric local time of $X$ at the origin.

By classical results on the pathwise uniqueness of reflecting Brownian
motion on $[0,\infty )$, $\left\vert X\right\vert $ is pathwise unique
(it is the reflecting Brownian motion on $[0,\infty )$ with driving Brownian
motion $B$).

So, although the process $X$ is not uniquely determined by $B$, its absolute
value $\left\vert X\right\vert $ is. We introduce the following.

\begin{definition}
\label{varphi strong solution}Consider $\varphi :\mathbb{R\rightarrow R}$ a
measurable function. A \emph{$\varphi $-strong solution} of the stochastic differential equation (\ref{SDE}) on a probability space $\left(
\Omega ,\mathcal{F},P\right) $ with respect to the Brownian motion $B$
is a continuous process $X$ for which (\ref{SDE}) holds a.s., and such
that the process $\varphi \left( X\right) $ is adapted to the augmented filtration
generated by $B$ and $P\left( \int_{0}^{t}(\left\vert b\left(
X_{s}\right) \right\vert +\sigma ^{2}\left( X_{s}\right)) ds<\infty \right)
=1 $ holds for every $t\geq 0$.

We say that \emph{$\varphi $-strong uniqueness} holds for (\ref{SDE}) if whenever $%
X$ and $\widetilde{X}$ are two $\varphi $-strong solutions relative
to the same driving Brownian motion $B$, we have $$P\left( \varphi \left(
X_{t}\right) =\varphi \big( \widetilde{X}_{t} \big) ,\quad t\geq 0\right)
=1.$$
\end{definition}

Similar definitions can be given for a $\varphi$-weak solution, $\varphi $-pathwise and $\varphi $-uniqueness in the sense of probability law.

Note that in the previous definition, the process $X$ is not required to
be adapted to the augmented filtration of the Brownian motion $B$. If
$X$ is adapted to the filtration $\left(\mathcal{F}_t\right)_{t\ge 0}$, where $\mathcal{F}_t$ is the augmented $\sigma $-algebra genera\-ted by $%
\sigma \left( B_{s}:s\leq t\right) \cup \mathcal{G}_{t}$, then the collection of events in $\mathcal{G}_{t}$ can be viewed as the extra source of randomness
needed in order to be able to predict the output $X_t$ from (\ref{SDE}) corresponding to a given input $\left(B_s\right)_{0\le s \le t}$.

\begin{remark}
It is easy to see that the above definitions are a natural extension of the
classical notions of strong solution and strong uniqueness, in the sense that
if $\varphi :\mathbb{R\rightarrow R}$ is an injective function,
then a $\varphi $-strong solution is the same as a strong solution and $%
\varphi $-strong uniqueness is the same as strong uniqueness for (\ref{SDE}).
\end{remark}

\begin{remark}
The above discussion shows that even though pathwise uniqueness and the
existence of a strong solution does not hold for SDE (\ref{SDE1}), $\left\vert
x\right\vert $-strong uniqueness holds for it, and we will also show that a $\left\vert x\right\vert $-strong solution exists.
\end{remark}

To do this, we will consider the more general case of the SDE%
\begin{equation}
X_{t}=\int_{0}^{t}\sigma _{a,b}\left( X_{s}\right) dB_{s},\qquad t\geq 0,
\label{SDE2}
\end{equation}%
where $\sigma _{a,b}:\mathbb{R\rightarrow R}$ is given by%
\begin{equation}
\sigma _{a,b}\left( x\right) =\left\{
\begin{tabular}{ll}
$a,$ & $x\geq 0$ \\
$b,$ & $x<0$%
\end{tabular}%
\right. ,  \label{sigma_a,b}
\end{equation}%
and $a,b\in \mathbb{R}^{\ast }=\mathbb{R}-\{0\}$ are arbitrary constants.

We first introduce the notion of \emph{sign choice} for a non-negative
process, as follows.

\begin{definition}\label{sign choice}
A \emph{sign choice} for a given a non-negative continuous process $\left(
Y_{t}\right) _{t\geq 0}$ is a process $\left( U_{t}\right) _{t\geq 0}$ taking the values $\pm 1$, such that $\left( U_{t}Y_{t}\right) _{t\geq 0}$ is a continuous
process.
\end{definition}

\begin{remark}
The condition that $Y$ is a non-negative process is not essential in the above definition and it can be dispensed off. It is only meant to suggest that the sign of the process $UY$ at time $t\ge 0$ is equal to $U_t$, thus justifying the name of the process $U$.
\end{remark}

The trivial choices $U \equiv 1$ and $U \equiv -1$ in the above definition give $U Y\equiv Y$, respectively the reflection $- Y$ of $Y$ with respect to the origin. The interesting cases of the above definition are those for which the sign choice $U$ takes both $\pm 1$ values  with positive probabilities, so the process $U Y$ is a combination between the original process $Y$ and the process $-Y$. Among these choices is the \emph{i.i.d. sign choice}, which corresponds to the intuitive idea of flipping a biased coin in order to choose the sign $U_t$ of the excursion of $Y$ straddling $t$. We will make this definition precise below.

Given a non-negative continuous adapted process $(Y_t)_{t\geq 0}$ with $Y_0=0$ on a filtered probability  space $(\Omega, \mathcal{F},P, (\mathcal{F}_t)_{t \geq 0})$, consider $${Z}={Z}^Y(\omega)=\{t\geq 0 :Y_t(\omega)=0\}$$ the zero set of $Y(\omega)$, $\omega\in\Omega$ (from now on, we will drop the dependence on $\omega\in \Omega$ whenever this will not create confusion). Since $Y$ is continuous, ${Z}^c=(0,\infty)-{Z}$ is an open set, hence it consists of a countable disjoint union (possibly finite, or even empty) of open intervals, called the \emph{excursion intervals} of $Y$ away from $0$. Denote by $\mathcal{I}$ the set of all excursion intervals of $Y$ away from zero, so ${Z}^c=\cup_{I\in \mathcal{I}} I$.

Next, we will order and index the excursions intervals in $\mathcal{I}$ according to their length and position, as follows. Consider the $(\mathcal{F}_t)_{t\ge 0}$-stopping times defined recursively by $\xi_0=0$ and
\begin{equation}
\xi_{i+1}=\inf\{t\geq \xi_i+1: Y_t=0\},\qquad i \geq 1,
\label{stopping times tau_m}
\end{equation}
where as usual $\inf \varnothing = \infty$. Note that it is possible that $\xi_{i+1}= \infty$, in which case $\xi_j=\infty$ for all $j\geq i+1$.

Consider the sets $\mathcal{I}_i=\{I\in \mathcal{I}: I\subset (\xi_{i-1},\xi_{i})\}$, $i\geq 1$, and note that by the definition of the stopping times $\xi_i$, $\{\mathcal{I}_i\}_{i\geq1}$ is a partition of $\mathcal{I}$. It is possible that some of the sets in this partition are empty: if $j=\min \{ i: \xi_i=\infty\}$, then $\xi_k=\infty$ for $k\geq j$, so $\mathcal{I}_k=\varnothing $ for all $k > j$.

For an excursion interval $I=(a,b)\in \mathcal{I}$ we will denote by $\vert I\vert=b-a$ the length of $I$, and by $I^l=a$ the left endpoint of $I$.

We introduce an order relation $\preceq$ on $\mathcal{I}$ by defining $I \preceq \widetilde{I}$ if $\vert I \vert >\vert \widetilde{I}\vert$, or $\vert I \vert = \vert \widetilde{I} \vert$ and $I^l < \widetilde{I}^l$. It is not difficult to see that $(\mathcal{I},\preceq )$ is a totally ordered set.

For any arbitrary fixed $i\geq 1$ for which $\mathcal{I}_i \neq \varnothing$, $(\mathcal{I}_i, \preceq )$ is a well-ordered set. To see this, note that by construction the sum of the lengths of the excursion intervals in $\mathcal{I}_i$ is bounded above by $\xi_{i}-\xi_{i-1}$ (since $\mathcal{I}_i\neq \varnothing$, $\xi_{i-1}\neq \infty$, so $\xi_{i-1}, \xi_{i}$ cannot be both infinite). If $\xi_{i} > \xi_{i-1}+1$, then except for the last interval excursion in $\mathcal{I}_i$ (the interval with right endpoint $\xi_i$, which possibly has infinite length), the sum of the lengths of the excursions intervals in $\mathcal{I}_i$ is at most $1$. This shows that for any $\varepsilon >0$, there are only finitely many interval excursions $I\in \mathcal{I}_i$ of length $\vert I \vert \geq \varepsilon $. Consider now a non-empty subset $\mathcal{S}\subset \mathcal{I}_i$, and let $I\in \mathcal{S}$. Since there are only finitely many interval excursions $\widetilde{I}\in \mathcal{S}$ with length $\vert \widetilde{I}\vert \geq \vert I\vert$, and since $\preceq$ is a total order on $\mathcal{I}$, $\mathcal{S}$ must have a least element.

We showed that for any $i\geq 1$ for which $\mathcal{I}_i\neq \varnothing$, $(\mathcal{I}_i,\preceq)$ is a well-ordered set. Since $\mathcal{I}_i$ is also a countable set, we can uniquely index its elements by positive integers such that $\mathcal{I}_i=\{I_{i,j}:1\leq j \leq N(i)\}$ for some $N(i)\in \mathbb{N}\cup \{\infty\}$, and
\begin{equation}
I_{i,1}\preceq \ldots \preceq I_{i,N(i)}.
\end{equation}

In order to simplify the notation (the dependence of $N(i)$), in the case when $\mathcal{I}_i\neq \varnothing$ is finite we define $I_{i,j}=\varnothing$ for $j>N(i)$, and we extend the order relation by defining $I\preceq \varnothing$ for any interval $I\in \mathcal{I}$. Also, in the case $\mathcal{I}_i=\varnothing$ for some $i\geq 1$, we define $I_{i,j}=\varnothing$ for all $j\geq 1$.

The above discussion shows that for any $\omega\in \Omega$ we can uniquely write the complement ${Z}^c (\omega)$ of the zero set of $Y(\omega)$ as a disjoint union of excursion intervals
\begin{equation} \label{ordered excursions}
{Z}^c(\omega)=\{t\geq 0:Y_t>0\}=\bigsqcup_{i,j\geq 1} I_{i,j}(\omega),
\end{equation}
with $I_{i,1}(\omega) \preceq I_{i,2}(\omega) \preceq \ldots$ for any $i\geq 1$, and $I_{i,j}^l(\omega) < I_{i^\prime,j^\prime}^l(\omega)$ for any $1\leq i< i^\prime$ and $j,j^\prime \geq 1$ for which $I_{i,j}(\omega),I_{i^\prime,j^\prime}(\omega)\neq\varnothing$.

Finally, note that the above ordering of the excursion intervals of $Y$ induces a corresponding ordering of the excursions of $Y$. Following \cite{Revuz and Yor} (Chap. XII), we define the \emph{excursion of $Y$ straddling $t$} by $e_t(u)=1_{\{u\leq d_t-g_t\}} Y(g_t+u)$ if $d_t-g_t>0$ and $\delta$ otherwise, where $g_t=\sup\{s < t: Y_s=0\}$ is the last visit of $Y$ to zero before time $t$, and $d_t=\inf \{s\ge t: Y_s=0\}$ is the first visit of $Y$ to zero after time $t$. The above shows that the set $E$ of excursions of $Y$ can be partitioned into the disjoint sets $E_i=\{e_{i,j}:j\geq 1\}$, $i\geq 1$, where $e_{i,j}$ represents the excursion of $Y$ straddling the excursion interval $I_{i,j}$, that is $e_{i,j}=e_t$ for $t$ chosen such that $(g_t,d_t)=I_{i,j}$. We set $e_{i,j}\preceq e_{i^\prime,j^\prime}$ iff $I_{i,j}\preceq I_{i^\prime,j^\prime}$, and we define the \emph{length} $R(e_{i,j})$ of the excursion $e_{i,j}$ to be the length of its corresponding excursion interval $I_{i,j}$, that is $R(e_{i,j})=\vert I_{i,j}\vert$.

With this preparation we can give the following.

\begin{definition}\label{iid sign choice}
An \emph{i.i.d. sign choice} for a given a non-negative continuous process $\left(Y_{t}\right) _{t\geq 0}$ on a probability space $(\Omega, \mathcal{F}, P)$ is a sign choice $\left( U_{t}\right) _{t\geq 0}$ for $(Y_t)_{t\ge 0}$ in the sense of Definition \ref{sign choice}, with $U_t(\omega)=1$ for $t\in Z^Y(\omega)=\{s\ge 0: Y_s(\omega)=0\}$, and for which the sequence $(U_{i,j})_{i,j\geq 1}$ defined by
\begin{equation}\label{construction of iid sign choice}
U_{i,j}(\omega)= U_t(\omega),\qquad t\in I_{i,j}(\omega), \, \omega\in\Omega, \, i,j\ge 1,
\end{equation}
is a sequence of i.i.d. random variables on $(\Omega, \mathcal{F}, P)$ which is also independent of $Y$ ($(I_{i,j}(\omega))_{i,j\geq 1}$ are the excursion intervals of $Y(\omega)$ ordered as above).
\end{definition}

\begin{remark} If $U$ is a sign choice for $Y$, from Definition \ref{sign choice} it follows that $U$ is constant on each of the excursion intervals $I_{i,j}$, $i,j\ge 1$, and therefore the random variables $U_{i,j}$, $i,j\ge 1$, in the above definition are well defined. Note that if $Y_t=0$ for some $t\ge 0$, then $U_t Y_t =0$ for any choice of $U_t \in \{-1,1\}$; the condition $U_t=1$ when $Y_t=0$ in the above definition is therefore not essential, and it can be dispensed off. We have chosen this condition in order to insure that $\mathrm{sgn}(U_t Y_t) =U_t$ for all $t\ge 0$, which justifies the name for the process $U$ (recall that we are working with the right continuous choice of $\mathrm{sgn}$ function, $\mathrm{sgn}(0)=1$). Also note that with this normalization, there is a one-to-one correspondence between the set of  i.i.d. sign choices $U$ for a fixed process $Y$ and the set of i.i.d. sequences of random variables $(U_{i,j})_{i,j\geq 1}$ taking the values $\pm 1$, which are also independent of $Y$.

\end{remark}

\begin{example}
The simplest and also interesting example of an i.i.d. sign choice $U$ is the case when the process $Y$ is the reflecting Brownian motion on $[0,\infty)$ and the distribution of the random variables $U_{i,j}$ is $P(U_{i,j}=\pm 1)=\frac12$. In this case, it can be shown (the proof is embedded in the proof of Theorem \ref{theorem 1} below for $a=-b=1$), that the resulting process $U Y$ is a $1$-dimensional Brownian motion.
As it is known, given a $1$-dimensional Brownian motion $B$, its absolute value $\vert B\vert$ has the distribution of a reflecting Brownian motion on $[0,\infty)$. The above shows the converse of this: the i.i.d. sign choice is the tool needed to reconstruct (in distribution, not pathwise) the Brownian motion from the reflecting Brownian motion.
\end{example}

Under mild assumptions, the distribution of an i.i.d. sign choice at a stopping time is the same as the distribution of its defining sequence of random variables, as shown in the following proposition.

\begin{proposition}\label{proposition on stopping times of iid sign choice}
If $U$ is an i.i.d. sign choice for $Y$ and $\tau$ is an a.s. finite $\mathbb{F}^Y$-stopping time with $P(Y_\tau=0)=0$, the random variable $U_\tau$ is independent of $Y$, and has the same distribution as the random variables $(U_{i,j})_{i,j\ge 1}$ appearing in the Definition \ref{iid sign choice} of $U$.
\end{proposition}

\begin{proof}
In the notation of Definition \ref{iid sign choice}, we have
\begin{equation}
\{U_\tau =1\}=\{\tau\in {Z}\} \bigsqcup \Bigg(\bigsqcup_{i,j\geq 1} \{U_{i,j}=1, \tau \in I_{i,j}\}\Bigg).
\end{equation}

Recalling the definition (\ref{stopping times tau_m}) of the $\big(\mathcal{F}^Y_t\big)_{t\ge 0}$-stopping times $\xi_i$,  it is not difficult to see that the event $\{\tau \in I_{i,j}\}$ is an event in $\mathcal{F}^Y_{\xi_i}$: on the set $\{\xi_i\leq t\}$, one ``knows'' the ordering of the excursion intervals $(I_{i,j}) _{j\geq 1}$ in $(\xi_{i-1},\xi_i)$ by time $t$, so $\{\tau \in I_{i,j}\}\cap \{\xi_i\leq t\}\in \mathcal{F}^Y_t$.

Since the random variables $(U_{i,j})_{i,j\ge 1}$ are by definition independent of $\mathcal{F}^Y_\infty$, and $P(\tau\in {Z}^Y)=P(Y_\tau=0)=0$ by hypothesis, we obtain
\begin{eqnarray*}
P(U_\tau =1)&=&\sum_{i,j=1}^\infty P(U_{i,j}=1) P(\tau \in I_{i,j})\\
&=&P(U_{1,1}=1) \sum_{i,j=1}^\infty P(\tau \in I_{i,j})\\
&=&P(U_{1,1}=1) P(\tau <\infty, \tau\in {Z}^c)\\
&=&P(U_{1,1}=1),
\end{eqnarray*}
so $U_\tau$ has the same distribution as the random variables $(U_{i,j})_{i,j\ge 1}$. A similar proof shows that $U_\tau$ is also independent of $\big(\mathcal{F}^Y_t\big)_{t \ge 0}$, concluding the proof.
\end{proof}

In the particular case when $Y$ is a continuous non-negative process for which $P(Y_t=0)=0$ for all $t>0$ and  considering $\tau\equiv t>0$ in the previous proposition, it follows that for all $t>0$ the random variable $U_t$ has the same distribution as the defining sequence $(U_{i,j})_{i,j\ge 1}$ in the Definition \ref{iid sign choice} of $U$. If moreover $P(U_{i,j}=1)=1-P(U_{i,j})=p$, $i,j\ge 1$, we will say that $U$ is an i.i.d. sign choice for $Y$ which \emph{takes the value $1$ with probability $p$}.

We can now prove the following.

\begin{theorem}\label{theorem 1}
Let $a>0>b$ be arbitrarily fixed, and consider the function $\varphi _{a,b}:\mathbb R\rightarrow \mathbb R$ defined by
\begin{equation}\label{definition of varphi}
\varphi _{a_{,}b}\left( x\right) =
\left\{
\begin{tabular}{ll}
$\frac{1}{a}x,$ & $x\geq 0$ \vspace{3pt}\\
$\frac{1}{b}x,$ & $x<0$%
\end{tabular}%
\right. .
\end{equation}

The following assertions are true relative to the stochastic differential equation (\ref{SDE2}).
\begin{enumerate}
\item (Existence) There exists a weak solution $(X,B,\mathbb{F})$, and every such solution is also a $\varphi_{a,b}$-strong solution.

\item (Uniqueness) $\varphi_{a,b}$-strong uniqueness holds.

\item (Construction) A weak solution and a $\varphi_{a,b}$-strong solution can be constructed from a given Brownian motion $(B,\mathbb{F}^B)$ as follows: $\big( \sigma _{a,b}\left( U\right)Y,B,\mathbb{G}\big) $, where $Y$ is the reflecting Brownian motion on $[0,\infty )$ with driving Brownian motion $B$, $U$ is an i.i.d. sign choice for $Y$ taking the value $1$ with probability $\frac{b}{b-a}$, and $\mathbb{G}$ is certain enlargement of the filtration $\mathbb{F}^B$ which satisfies the usual hypothesis (see Appendix \ref{Appendix}).

\item (Representation) Every weak solution $(X,B,\mathbb{F})$ has the representation $X=\sigma_{a,b}(U) Y$, where $Y$ and $U$ are as above.
\end{enumerate}
\end{theorem}

\begin{proof}
That (\ref{SDE2}) has a weak solution follows immediately from classical results (the set $I(\sigma_{a,b})$ of non-local integrability of $\sigma_{a,b}^{-2}$ is in this case empty, and coincides with the set $Z(\sigma_{a,b})$ of zeroes of $\sigma_{a,b}$, so by the results of Engelbert and Schmidt \cite{Engelbert-Schmidt} there exists a weak solution of (\ref{SDE2}), which is also unique in the sense of probability law).

If $(X,B,\mathbb{F})$ is a weak solution of (\ref{SDE2}), its quadratic variation process is given by
\[
\langle X\rangle _{t}=\int_{0}^{t}\sigma _{a,b}^{2}\left( X_{s}\right)
ds=\int_{0}^{t}\left (a^{2}1_{[0,\infty )}\left( X_{s}\right) +b^{2}1_{(-\infty
,0)}\left( X_{s}\right)\right) ds.
\]%

In particular, the above shows that $d\langle X\rangle _{t}$ is absolutely continuous with
respect to the Lebesgue measure, so by Corollary VI.I.6 of \cite{Revuz and Yor}, the Lebesgue measure of the time spent by $X$ at the origin is zero. Using Corollary VI.1.9 of \cite{Revuz and Yor} and the fact that $X$ is a local martingale, it follows that $L_t^0(X)=L_t^0(-X)$, where $L_{t}^{0}\left( X\right) =\lim_{\varepsilon \searrow 0}\frac{1}{\varepsilon }\int_{0}^{t}1_{\left[0 ,\varepsilon \right)}\left( X_{s}\right) d\langle X\rangle _{s}$ denotes the semimartingale local time of $X$ at the origin.

Applying the Tanaka-It\^{o} formula to
the convex function $\varphi _{a,b}$ and to the process $X$ we obtain%
\begin{eqnarray}
\varphi _{a,b}\left( X_{t}\right) &=&\int_{0}^{t}\varphi _{a,b}^{{\prime
}}\left( X_{s}\right) dX_{s}+\frac{1}{2}\left( \varphi _{a,b}^{\prime
}\left( 0+\right) -\varphi _{a,b}^{\prime }\left( 0-\right) \right)
L_{t}^{0}\left( X\right) \label{aux1}\\
&=&B_{t}+\frac{1}{2}\left( \frac{1}{a}-\frac{1}{b}%
\right) L_{t}^{0}\left( X\right) ,  \notag
\end{eqnarray}%
which shows that the process $Y=\varphi _{a,b}\left( X\right) $ is a semimartingale with quadratic variation $\langle Y\rangle_t=t$. Note
that $1_{\left[0,\varepsilon \right) }\left( Y_{s}\right)
=1_{[0,a\varepsilon )}\left( X_{s}\right) +1_{(b\varepsilon ,0)}\left(
X_{s}\right) $ for any $\varepsilon >0$ and $t\geq 0$.  It follows that the semimartingale local time of $Y$ at the origin is given by
\begin{eqnarray*}
L_{t}^{0}\left( Y\right) &=&\lim_{\varepsilon \searrow 0}\frac{1}{%
\varepsilon }\int_{0}^{t}1_{\left[ 0 ,\varepsilon \right)
}\left( Y_{s}\right) d\langle Y\rangle _{s} \\
&=&\lim_{\varepsilon \searrow 0}\frac{1}{\varepsilon }\int_{0}^{t}1_{[0,a%
\varepsilon )}\left( X_{s}\right) ds+\lim_{\varepsilon \searrow 0}\frac{1}{%
2\varepsilon }\int_{0}^{t}1_{(b\varepsilon ,0)}\left( X_{s}\right) ds \\
&=&\lim_{\varepsilon \searrow 0}\frac{1}{\varepsilon }\int_{0}^{t}1_{[0,a%
\varepsilon )}\left( X_{s}\right) \frac{1}{a^{2}}d\langle X\rangle
_{s}+\lim_{\varepsilon \searrow 0}\frac{1}{\varepsilon }\int_{0}^{t}1_{(b%
\varepsilon ,0)}\left( X_{s}\right) \frac{1}{b^{2}}d\langle X\rangle _{s} \\
&=&\frac{1}{a}L_t^0(X)-\frac{1}{b} L_{t}^{0}\left( -X\right)\\
&=&\left( \frac{1}{a}-\frac{1}{b}\right) L_{t}^{0}\left( X\right)
\end{eqnarray*}

From (\ref{aux1}) it follows that the process $Y=\varphi _{a,b}\left(
X\right) $ satisfies the SDE%
\[
Y_{t}=B_{t}+\frac12L_{t}^{0}\left( Y\right) ,\qquad t\geq 0,
\]%
and therefore $Y$ is the reflecting Brownian motion on $%
[0,\infty )$ with driving Brownian motion $B$. In particular, the
process $Y=\varphi _{a,b}\left( X\right) $ is adapted to the filtration $\mathbb{F}^{B}$ of the Brownian motion $B$ and it is
pathwise unique. This proves the last part of the first assertion of the theorem and the second one.

To prove the third claim of the theorem, consider a $1$-dimensional $(B,\mathbb{F}^B)$ starting at the origin on the canonical probability space $(C[0,\infty),\mathcal{B}(C[0,\infty)),P)$ and let $Y$ be the reflecting Brownian motion on $[0,\infty )$ with driving Brownian motion $B$. Consider the process $X$ defined by $X_{t}=\sigma _{a,b}\left( U_{t}\right) Y_{t}$, $t\ge 0$, where $U$ is an i.i.d. sign choice for $Y$ taking the value $1$ with probability $\frac{b}{b-a}$, and $\mathbb{G}$ is the corresponding filtration constructed in the Appendix \ref{Appendix}, under which $(UY,\mathbb{G})$ is a skew Brownian motion with parameter $\alpha=\frac{b}{b-a}$.

We will now show that $\left( X,B,\mathbb{G}\right)$ is a weak solution and also a $\varphi _{a,b}$-strong solution of (\ref{SDE2}). First, we will show that $B$ remains a Brownian motion under the enlarged filtration $\mathbb{G}$. To see this, note that by the construction (see Lemma \ref{equivalence of excursion numberings} and the concluding remarks in Appendix \ref{Appendix} for the details), the process $X^\alpha$ defined by $X^\alpha_t=U_t Y_t$, $t\ge 0$, is a skew Brownian motion with parameter $\alpha=\frac{b}{b-a}$, and moreover $X^\alpha$ is a $\mathbb{G}$-adapted strong Markov process. Since $Y=\vert X^\alpha\vert$, the process $Y$ is also a $\mathbb{G}$-adapted strong Markov process. The representation $$L_{t+s}^0 (Y)=L_s^0(Y)+\lim_{\varepsilon\rightarrow 0} \frac{1}{\varepsilon}\int_0^{t} 1_{[0,\varepsilon)}(Y_{s+u}) du=L_s^0(Y)+L_t^0(Y\circ\theta_s), \qquad t,s,\ge 0,$$  shows that $L^0(Y)$ is also a $\mathbb{G}$-adapted strong Markov process, and therefore so is $B=Y-\frac12L^0(Y)$. We obtain $E(B_{t+s}\vert \mathcal{G}_s)=E^{B_s} B_t=B_s$, since the distribution of $B$ starting at $x\in\mathbb R$ is normal with mean $x$ (recall that $(B,\mathbb{F}^B)$ is a Brownian motion).


The above shows that $(B,\mathbb{G})$ is a martingale, and since $B$ is also continuous and has quadratic variation at time $t\ge 0$ equal to $t$ (since $\left(B,\mathbb{F}^B\right)$ is a Brownian motion), by L\'{e}vy's characterization of Brownian motion it follows that $(B,\mathbb{G})$ is a Brownian motion.

Next, note that by the definition (\ref{definition of varphi}) of $\varphi_{a,b}$ we have  $\varphi_{a,b}(X_t)=Y_t \varphi_{a,b}(\sigma_{a,b}(U_t))=Y_t$, and since $Y$ is adapted to the filtration $\mathbb{F}^B$ generated by $B$, so is $\varphi_{a,b}(X)$. In order to prove the claim it remains therefore to show that $X$ verifies the SDE (\ref{SDE2}).

Recalling the construction (\ref{construction of iid sign choice}) of the i.i.d. sign choice $U$ for $Y$, we see that $U_t=1$ when $Y_t=0$, so $\mathrm{sgn}\left( X_{t}\right)= \mathrm{sgn}\left( \sigma _{a,b}\left( U_{t}\right) Y_{t} \right) =U_{t}$, and therefore $\sigma _{a,b}\left( X_{t}\right) =\sigma
_{a,b}\left( U_{t}\right)$ for any $t\geq 0$. Since the process $Y$ (hence $X$) spends zero Lebesgue time at the origin, we have
almost surely%
\begin{equation*}
\int_{0}^{t}\sigma _{a,b}\left( X_{s}\right) dB_{s}=\int_{0}^{t}\sigma
_{a,b}\left( U_{s}\right) 1_{\mathbb{R}^{\ast }}\left( X_{s}\right) dB_{s},\qquad t\ge 0.
\end{equation*}

Since $Y$ is reflecting Brownian motion with driving Brownian motion $B$, we have
\begin{eqnarray*}
\int_{0}^{t}\sigma _{a,b}\left( X_{s}\right) dB_{s} &=&\int_{0}^{t}\sigma
_{a,b}\left( U_{s}\right) 1_{\mathbb{R}^{\ast }}\left( X_{s}\right)
dY_{s}-\frac12\int_{0}^{t}\sigma _{a,b}\left( U_{s}\right) 1_{\mathbb{R}^{\ast }}\left(
X_{s}\right) dL_{s}^{0}\left( Y\right) \\
&=&\int_{0}^{t}\sigma _{a,b}\left( U_{s}\right) 1_{\mathbb{R}^{\ast }}\left(
X_{s}\right) dY_{s},
\end{eqnarray*}%
where the last equality follows from the fact that the semimartingale local time $L^{0}\left(
Y\right) $ of $Y$ at the origin increases only when $Y$ (hence $X$) is at the origin.

\label{pagina de start referinta}
For $\varepsilon >0$ arbitrarily fixed, consider the $\mathbb{F}^Y$-stopping times $\tau
_{i}$ and $\sigma _{i}$ defined inductively by $\tau _{0}=0$,%
\[
\sigma _{i}=\inf \left\{ s\geq \tau _{i-1}:Y_{s}=\varepsilon \right\} \quad
\text{and\quad }\tau _{i}=\inf \left\{ s\geq \sigma _{i}:Y_{s}=0\right\} ,\qquad
i\geq 1.
\]

Denoting by $D_{t}\left( \varepsilon \right) =\sup \left\{ i\geq 0:\tau
_{i}\leq t \right\} $ the number of downcrossings of the interval $\left[
0,\varepsilon \right] $ by the process $(Y_{s})_{0\leq s\leq t}$, we obtain
\begin{eqnarray*}
&&\int_{0}^{t}\sigma _{a,b}\left( X_{s}\right) dB_{s} =\\
&=&\int_{0}^{t}\sigma
_{a,b}\left( U_{s}\right) 1_{\mathbb{R}^{\ast }}( X_{s}) dY_{s} \\
&=&\sum_{i\geq 1}\int_{\sigma _{i}\wedge t}^{\tau _{i}\wedge t}\sigma
_{a,b}\left( U_{s}\right) 1_{\mathbb{R}^{\ast }}( X_{s})
dY_{s}+\sum_{i\geq 1}\int_{\tau _{i-1}\wedge t}^{\sigma _{i}\wedge t}\sigma
_{a,b}\left( U_{s}\right) 1_{\mathbb{R}^{\ast}}( X_{s}) dY_{s} \\
&=&\sum_{i\geq 1}\sigma _{a,b}\left( U_{\sigma _{i}\wedge t}\right) \left(
Y_{\tau _{i}\wedge t}-Y_{\sigma _{i}\wedge t}\right) +\sum_{i\geq
1}\int_{\tau _{i-1}\wedge t}^{\sigma _{i}\wedge t}\sigma _{a,b}\left(
U_{s}\right) 1_{\mathbb{R}^{\ast}}( X_{s}) dY_{s} \\
&=&-\varepsilon \sum_{i=1}^{D_{t}\left( \varepsilon \right) }\sigma
_{a,b}\left( U_{\sigma _{i}}\right) +\sigma _{a,b}\left( U_{t}\right) \left(
Y_{t}-\varepsilon \right) \sum_{i\geq 1}1_{[\sigma _{i},\tau _{i})}\left(
t\right)\\
&&+\sum_{i\geq 1}\int_{\tau _{i-1}\wedge t}^{\sigma _{i}\wedge
t}\sigma _{a,b}\left( U_{s}\right) 1_{\mathbb{R}^{\ast }}\left( X_{s}\right) dY_{s},
\end{eqnarray*}%
and therefore%
\begin{align}
\int_0^t \sigma_{a,b} \left( X_{s}\right) d B_{s}-\sigma _{a,b}\left(
U_{t}\right) Y_{t} =&-\varepsilon \sum_{i=1}^{D_{t}\left( \varepsilon
\right) }\sigma _{a,b}\left( U_{\sigma _{i}}\right) \label{aux2} \\
&-\sigma _{a,b}\left(
U_{t}\right) Y_{t}\sum_{i\geq 1}1_{[\tau _{i-1},\sigma _{i})}\left( t\right)& \nonumber\\
&-\varepsilon \sigma \,_{a,b}\left( U_{t}\right) \sum_{i\geq 1}1_{[\sigma
_{i},\tau _{i})}\left( t\right)& \nonumber\\
&+\sum_{i\geq 1}\int_{\tau _{i-1}\wedge
t}^{\sigma _{i}\wedge t}\sigma _{a,b}\left( U_{s}\right) 1_{\mathbb{R}^{\ast }}\left(
X_{s}\right) dY_{s}.&  \nonumber
\end{align}

To prove the claim, we will show that all the terms on the right of the above
equality converge to zero in $L^{2}$ as $\varepsilon \searrow 0$.

By construction, $\sigma_i$ are a.s. finite $(\mathcal{F}^Y_t)$-stopping times, with $Y_{\sigma_i}=\varepsilon\neq0$ a.s. for any $i\geq1$, so they satisfy the hypotheses of Proposition \ref{proposition on stopping times of iid sign choice}. It follows that $\left( U_{\sigma _{i}}\right) _{i\geq 1}$ are identically distributed random variables, taking the values $\pm 1$ with probability $P\left( U_{\sigma _{i}}=1\right) =1-P\left( U_{\sigma _{i}}=-1\right) =\frac{b%
}{b-a}$ for any $i\geq 1$.

Also note that by the construction of the stopping times $\sigma_i$, the process $Y$ visits the origin in each of the random time intervals $(\sigma_i, \sigma_{i+1})$, hence for each $1\leq i <j$, $\sigma_i$ and $\sigma_j$ belong to different excursion intervals of $Y$ away from zero. Recalling the definition (\ref{construction of iid sign choice}) of the i.i.d. sign choice $U$ for $Y$, and using the fact that the random variables $(U_{i,j})_{i,j\geq1}$ defining $U$ are i.i.d. and also independent of $\mathcal{F}^ Y$, for any $u,v\in\{\pm1\}$ and any $1\leq i < j$ we obtain
\begin{eqnarray*}
P(U_{\sigma_i}=u, U_{\sigma_j}=v)&=&\sum_{(m,n)\neq (p,q)}P\left(U_{m,n}=u,U_{p,q}=v, \sigma_i\in I_{m,n}, \sigma_j\in I_{p,q}\right)\\
&=&\sum_{(m,n)\neq (p,q)}P(U_{m,n}=u)P(U_{p,q}=v) P(\sigma_i\in I_{m,n}, \sigma_j\in I_{p,q})\\
&=& P(U_{1,1}=u)P(U_{1,1}=v) \sum_{(m,n)\neq (p,q)}P(\sigma_i\in I_{m,n}, \sigma_j\in I_{p,q})\\
&=& P(U_{\sigma_i}=u)P(U_{\sigma_j}=v) P((\sigma_i, \sigma_j) \cap {Z}^Y \neq \varnothing)\\
&=& P(U_{\sigma_i}=u)P(U_{\sigma_j}=v).
\end{eqnarray*}

This shows that for any $1\leq i<j$, the random variables $U_{\sigma_i}$ and $U_{\sigma_j}$ are independent. A similar proof shows that $(U_{\sigma_i})_{i\geq 1}$ forms an independent sequence of random variables, hence $(U_{\sigma_i})_{i\geq 1}$ is an i.i.d. sequence of random variables.

From the definition of the function $\sigma
_{a,b} $ it follows that $\left( \sigma _{a,b}\left( U_{\sigma _{i}}\right)
\right) _{i\geq 1}$ are independent random variables with mean $0$ and
variance $-ab$, so using Wald's second identity we obtain%
\[
E\left( -\varepsilon \sum_{i=1}^{D_{t}\left( \varepsilon \right) }\sigma
_{a,b}\left( U_{\sigma _{i}}\right) \right) ^{2}=\varepsilon ^{2}E\left(
\sigma _{a,b}^2\left( U_{\sigma _{1}}\right) \right) ED_{t}\left( \varepsilon
\right) =-ab\varepsilon ^{2}ED_{t}\left( \varepsilon \right) \rightarrow 0
\]%
as $\varepsilon \searrow 0$, since by L\'evy's characterization of the local
time we have $\varepsilon D_{t}\left( \varepsilon \right) \rightarrow
L_{t}^{0}\left( Y\right) $ a.s. and also in $L^{2}$ (see \cite%
{Karatzas-Shreve}, p. 416). This shows the $L^2$ convergence to zero as $%
\varepsilon \searrow 0$ of the first term on the right of (\ref{aux2}).

Next, note that if $t\in \lbrack \tau _{i-1},\sigma _{i})$ for some $i\geq 1$%
, by construction we have $Y_{t}\in \lbrack 0,\varepsilon )$, so we obtain%
\begin{eqnarray*}
E\left( \sigma _{a,b}\left( U_{t}\right) Y_{t}\sum_{i\geq 1}1_{[\tau
_{i-1},\sigma _{i})}\left( t\right) \right) ^{2}&\leq & \max \left\{
a^2,b^2 \right\} \varepsilon ^{2}E\sum_{i\geq 1}1_{[\tau _{i-1},\sigma
_{i})}\left( t\right) \\
&\leq & \max \left\{ a^2,b^2\right\} \varepsilon
^{2}\\
&\rightarrow& 0
\end{eqnarray*}
as $\varepsilon \searrow 0$. This proves the $L^2$ convergence to zero as $%
\varepsilon \searrow 0$ of the second term on the right of (\ref{aux2}). For
the third term the proof being similar, we omit it.

Using again Wald's second identity and the fact
that $\sigma _{i}-\tau _{i-1}$, $i=1,2,\ldots $ are
i.i.d. random variables (see Theorem 2.6.16 in \cite{Karatzas-Shreve}), with
mean $E\left( \sigma _{1}-\tau _{0}\right) = E\sigma _{1}=\varepsilon ^{2}$, we obtain
\begin{eqnarray*}
E\left( \sum_{i\geq 1}\int_{\tau _{i-1}\wedge t}^{\sigma _{i}\wedge t}\sigma
_{a,b}\left( U_{s}\right) 1_{\mathbb{R}^{\ast }}\left( X_{s}\right) dY_{s}\right)
^{2} &\leq&\max \left\{ a^{2},b^{2}\right\} E\int_{0}^{t}\sum_{i\geq
1}1_{[\tau _{i-1},\sigma _{i}]}\left( s\right) d\langle Y\rangle _{s} \\
&\leq &\max \left\{ a^{2},b^{2}\right\} E\sum_{i=1}^{D_{t}\left( \varepsilon
\right) +1}\left( \sigma _{i}-\tau _{i-1}\right) \\
&=&\max \left\{ a^{2},b^{2}\right\} E\left( \sigma _{1}-\tau _{0}\right)
E\left( D_{t}\left( \varepsilon \right) +1\right) \\
&=&\max \left\{ a^{2},b^{2}\right\} \varepsilon ^{2}E\left( D_{t}\left(
\varepsilon \right) +1\right) \\
&\rightarrow &0,
\end{eqnarray*}%
and therefore the last term on the right side of (\ref{aux2}) also converges
to zero in $L^2$ as $\varepsilon \searrow 0$.

We have shown that all the terms on the right side of (\ref{aux2}) converge
to zero as $\varepsilon \searrow 0$. Passing to the limit with $\varepsilon
\searrow 0$ in (\ref{aux2}) we obtain $\int_{0}^{t}\sigma _{a,b}\left(
X_{s}\right) dB_{s}=\sigma _{a,b}\left( U_{t}\right) Y_{t}=X_{t}$,
concluding the proof of the third claim of the theorem.\label{pagina de sfarsit referinta}

To prove the last claim of the theorem, note that if $\left( X,B,\mathbb{F}\right)$ is a weak solution of (\ref{SDE2}), then by the first part of the proof it follows that $Y=\varphi_{a,b}(X)$ is the reflecting Brownian motion on $[0,\infty)$ with driving Brownian motion $B$.

From the definition of the functions $\sigma_{a,b}$ and $\varphi_{a,b}$ it follows that $X$ can be written in the form
\begin{equation}\label{representation no 1}
X_t=\sigma_{a,b}(U_t)Y_t, \qquad t\geq 0,
\end{equation}
where $U_t=\mathrm{sgn}(X_t)=1_{[0,\infty)}(X_t)-1_{(-\infty,0)}(X_t)$.

The continuity of the processes $X$ and $Y$ (and the discontinuity at the origin of $\sigma_{a,b}$) shows that $U$ can only change signs when $Y=0$, so $U Y$ is a continuous process, and therefore $U$ is a sign choice for $Y$ in the sense of Definition \ref{sign choice}.

To show that $U$ is an i.i.d. sign choice for $Y$, according to Definition \ref{iid sign choice} we have to show that the restriction of $U$ to the excursion intervals of $Y$ forms an i.i.d. sequence, also independent of $Y$. To do this, consider the ordered excursion intervals $(I_{i,j})_{i,j\geq 1}$ of $Y$ away from zero, as defined before Definition \ref{iid sign choice}, and define the random variables $(U_{i,j})_{i,j\geq 1}$ by $U_{i,j}(\omega)=U_t(\omega)$ for $t\in I_{i,j}(\omega)$; $U_{i,j}$ is thus the sign of $X$ during the excursion interval $I_{i,j}$ of $Y$ (note that by (\ref{representation no 1}) the excursion intervals of $X$ and $Y$ are the same).

Since $Y$ is reflecting Brownian motion starting at the origin, we have $I_{i,j}\ne \varnothing$ a.s. for any $i,j\geq 1$ (recall that when ordering the excursion intervals of $Y$, if $\mathcal{I}_i\neq \varnothing $ is a finite set, we defined $I_{i,j}=\varnothing$ for $j>N(i)$; for the reflecting Brownian motion this is not the case, since starting at $0$ it visits the origin infinitely often, at times arbitrarily close to $0$). This, together with the fact that $X$ has constant sign during each excursion interval $I_{i,j}$, shows that the random variables $U_{i,j}$ introduced above are well defined up to a set of zero probability, on which we consider $U_{i,j}=1$.

Since $X$ is a weak solution of (\ref{SDE2}), by L\'{e}vy's characterization of Brownian motion it follows that the process $W$ defined by $W_t=X_{\alpha_t}$, $t\ge0$, is a Brownian motion, where the time change $\alpha_t$ is the inverse of the continuous increasing process $A$ given by $A_t=\langle X\rangle _t$, $t\ge 0$. Undoing the time change, $X$ may be written in the form
\begin{equation}\label{representation no 2}
X_t=W_{A_t}, \qquad t\geq 0,
\end{equation}
where $W$ is a $1$-dimensional Brownian motion starting at $0$ and $A_t=\int_0^t \sigma_{a,b}^2(X_s) ds$.

Comparing the representations (\ref{representation no 1}) and (\ref{representation no 2}) of $X$, it can be seen that the excursions of $Y$, space-scaled by $\sigma_{a,b}(U)$, coincide with the excursions of $W$, time-scaled by $\sigma_{a,b}^2 (X)$. More precisely, denoting by $e^Y_t$ the excursion of $Y$ straddling $t$, and by $e^W_{A_t}$ the excursion of $W$ straddling $A_t=\int_0^t \sigma_{a,b}^2 (X_s)ds$, we have
\begin{equation}\label{correspondence between excursions}
e^Y_t(u)=\frac{1}{\sigma_{a,b}(X_t)} e^W_{A_t} (\sigma_{a,b}^2(X_t) u)=\left \vert s_{a,b}\left(e^W_{A_t}\right) (u)\right\vert,
\end{equation}
where $s_{a,b}$ is the scaling map defined on the space of excursions by
\begin{equation}\label{scaling map s}
s_{a,b}(e)(u)=\frac{1}{\vert \sigma_{a,b}(e)\vert} e\left(\sigma_{a,b}^2(e) u\right), \qquad u\geq 0,
\end{equation}
and by an abuse of notation we write $\sigma_{a,b}(e)=a$ if $e$ is a positive excursion, and $\sigma_{a,b}(e)=b$ if $e$ is a negative excursion.

\begin{remark}
In the discussion preceding Definition \ref{iid sign choice}, we partitioned the excursion intervals and the corresponding excursions of $Y$ into the disjoint sets $\mathcal{I}_i=\{I_{i,j}: I_{i,j}\subset (\xi_{i-1},\xi_i)\}$, respectively $E_i=\{e^Y_t: \xi_{i-1}\leq t <\xi_i\}$, by using the stopping times defined by $\xi_0=0$ and $\xi_{i+1}=\inf \{t\geq \xi_i +1: Y_t=0\}$.

To partition the excursions of $W$ into disjoint sets $\widetilde{E}_i$ such that under (\ref{correspondence between excursions}) the correspondence of the excursions of $Y$ in $E_i$ and of $W$ in $\widetilde{E}_i=\{e^W_s: \widetilde{\xi}_{i-1}\leq s< \widetilde{\xi}_{i}\}$ is one-to-one, we use the same construction, but with the following modification. Instead of the stopping times $\xi_i$ we use the stopping times $\widetilde{\xi}_i$ of $W$ corresponding to $\xi_i$ under the time change $A$, that is $\widetilde{\xi}_0=0$ and $\widetilde{\xi}_{i+1}=\inf \{t\geq A_{\widetilde{\xi}_i +1}: W_t=0\}$. With this choice, there is a one-to-one correspondence under (\ref{correspondence between excursions}) between the excursions of $Y$ in $E_i$ and the excursions of $W$ in $\widetilde{E}_i$.
\end{remark}

Together with (\ref{correspondence between excursions})--(\ref{scaling map s}), the above shows that the excursions of $Y$ in $E_i$ are obtained from the excursions of $W$ in $\widetilde{E}_i$ by scaling the positive excursions by $a$ and the negative excursions by $b$, and by taking their absolute value. If $\{\widetilde e_{i,j} \}_{i,j\geq 1}$ are the ordered excursion of $W$ away from zero (using the stopping times $\widetilde{\xi}_i$ instead of $\xi_i$, as indicated in the above remark), note that under the above scaling of the excursions of $W$ the ordering $\preceq$ of the positive excursion is preserved, as well as the ordering of the negative excursions. However, if $a\neq \vert b \vert$ the ordering of positive and negative excursion is not preserved; for example, if $a>\vert b\vert$, and $i\geq 1$ and $1\leq j<k$ are such that $\widetilde{e}_{i,j}$ is a positive excursion and $\widetilde{e}_{i,k}$ is a negative excursion of $W$ with $1< R(\widetilde e_{i,j}) / R(\widetilde e_{i,k})< a^2 /b^2$, then $\widetilde e_{i,j} \preceq \widetilde e_{i,k}$ and after scaling the lengths of the excursions satisfy
\begin{equation*}
R(s_{a,b}(\widetilde e_{i,j}))=\frac{1}{a^2} R(\widetilde e_{i,j}) < \frac{1}{b^2}R(\widetilde e_{i,k})=R(s_{a,b}(\widetilde e_{i,k})),
\end{equation*}
so the order of the scaled excursion is reversed: $s_{a,b}(\widetilde e_{i,k})\preceq s_{a,b}(\widetilde e_{i,j})$.

If $(V_{i,j})_{i,j\geq 1}$ is the sequence of random variables representing the signs of $W$ during the corresponding excursions $(\widetilde e_{i,j})_{i,j\geq1}$, the above discussion shows that $U_{i,j}$ (the sign of $X$ during the $j^\text{th}$ longest excursion in $E_i$) is $V_{i,k}$ (the sign of $W$ during the $k^\text{th}$ longest excursion in $\widetilde E_i$), provided that after the scaling of the excursions of $W$, $\widetilde e_{i,k}$ becomes the $j^\text{th}$ longest excursion in $\widetilde E_{i}$.

We will first show that $(U_{1,n})_{n\geq 1}$ is an i.i.d. sequence.

In the following, we will assume and use the basic results and the notation of \cite{Revuz and Yor} (Chap. XII) concerning the It\^{o} excursion theory. Recall that the excursion process $\widetilde{e}=(\widetilde{e}_t)_{t>0}$ of the Brownian motion $W$ is a  $(\mathcal{F}_{\tau_t}^W)_{t> 0}$-Poisson point process ($\tau_t=\inf\{s> 0: L_s^0(W)>t\}$ is the right-continuous inverse of the local time $L^0(W)$ of $W$ at the origin), with characteristic measure
\begin{equation}
n (\Gamma)=\frac{1}{t} E(N_t^\Gamma), \qquad t>0,
\end{equation}
for each $\Gamma \in \mathcal{U}_\delta$ in the measurable space $({U}_\delta, \mathcal{U}_\delta)$ of excursions  of $W$, where $N_t^{\Gamma}(\omega)=\sum_{s\le t} 1_{(0,t]\times \Gamma}(s,\widetilde{e}_s(\omega)$ if $t>0$, and $N_0^\Gamma=0$.

The excursion length process $(R(\widetilde{e}_t))_{t>0}$ is also a $(\mathcal{F}_{\tau_t}^W)$-Poisson point process, with law under $n$ given by
\begin{equation}
n(R>x)=\frac1t E(N_t^{\{R>x\}})=\sqrt{\frac{2}{\pi x}}, \qquad t,x>0.
\end{equation}

Moreover, for each $\Gamma \in \mathcal{U}_\delta$ with $n(\Gamma)<\infty$, $\left(N_t^\Gamma-tn(\Gamma),\mathcal{F}_{\tau_t}\right)_{t\ge0}$ is a martingale. The local time at the origin of $W$ can be described as the right continuous inverse of $\tau_t$, that is $L_t^0(W)=\inf\{s>0:\tau_s>t\}$, so in particular it is a $(\mathcal{F}_{\tau_t})_{t\ge0}$-stopping time. In particular $t\wedge L_1$ is a bounded $(\mathcal{F}_{\tau_t})_{t\ge0}$-stopping time for any $t>0$, so by Doob's optional stopping theorem it follows that $$EN_{t\wedge L_1}^{\{R>x\}} =n(\{ R > x \}) E\left(t\wedge L_1\right)$$

Letting $t\nearrow \infty$ and using monotone convergence on the right side and bounded convergence on the left side (note that $N_{L_1}^{\{R>x\}}-1$ is the same as the number of excursions of $W$ away from the origin, with length greater than $x$ and completed by time $1$, and since the total length of these excursions is less than or equal to $1$, we have $N_{t\wedge L_1}^{\{R>x\}}-1 \le N_{L_1}^{\{R>x\}} -1\le \frac{1}{x}$), we obtain
\begin{equation*}
E N_{L_1}^{\{R>x\}}=\lim_{t\nearrow \infty} E N_{t \wedge L_1}^{\{R>x\}} = n(\{ R > x \})\lim_{t \nearrow \infty}  E\left(t\wedge L_1\right) = n(\{ R > x \}) E\left(L_1\right),
\end{equation*}
which shows that the intensity of the Poisson random variable  $N_{L_1}^{\{R>x\}}$ is also $n(\{R>x\})=\sqrt{\frac{2}{\pi x}}$.

Consider now the processes
\begin{equation}
N_1(x)=N_{L_1}^{\{R>a^2x^{-2}\}\cap \mathcal{U}_\delta^+} \quad\text{and}\quad N_2(x)=N_{L_1}^{\{R>b^2x^{-2}\}\cap \mathcal{U}_\delta^-}, \qquad x>0,
 \end{equation}
representing the number of positive, respectively negative excursions of $W$ starting before time $1$, with lengths greater that $(\frac{a}{x})^2$, respectively $(\frac{b}{x})^2$.

\label{N_1 and N_2}

Both $N_1$ and $N_2$ are nondecreasing, right-continuous processes with independent increments. For $0<x<y$, the increment $N_1(y)-N_1(x)$ is a Poisson random variable with expectation
\begin{eqnarray*}
E(N_1(y)-N_1(x))&=&E\left (N_{L_1}^{\{a^2y^{-2}<R<a^2x^{-2}\}\cap \mathcal{U}_\delta^+}\right)\\
&=&n_+(a^2y^{-2}<R<a^2x^{-2}) E L_1\\
&=&\frac12 n(a^2y^{-2}<R<a^2x^{-2}) EL_1\\
&=&\frac12\sqrt{\frac{2}{\pi}} \left(\frac{y}{a}-\frac{x}{a}\right) EL_1,
\end{eqnarray*}
and similarly
\begin{equation}
E(N_2(y)-N_2(x))=\frac12\sqrt{\frac{2}{\pi}} \left(\frac{y}{\vert b\vert}-\frac{x}{\vert b\vert}\right) EL_1,
\end{equation}
where $n_+$ and $n_-$ denote the restrictions of $n$ to the measurable space of positive excursions $(U_\delta^+,\mathcal{U}_\delta^+)$, respectively to the space of negative excursions $(U_\delta^+,\mathcal{U}_\delta^+)$, so $n_+(\Gamma)=n_-(\Gamma)=\frac12 n(\Gamma)$ for any $\Gamma \in \mathcal{U}_\delta$.

It follows that $(N_1(x))_{x>0}$ and $(N_2(x))_{x>0}$ are Poisson processes, with intensities
\begin{equation}
\lambda_1=\frac{1}{a}\sqrt{\frac1{2\pi}} EL_1\qquad \text{and} \qquad \lambda_2=\frac{1}{\vert b\vert}\sqrt{\frac1{2\pi}} E L_1 ,
\end{equation}
which may be written in the form $\lambda_1=p_1\lambda$ and $\lambda_2=p_2 \lambda$, where
\begin{equation}\label{p_1, p_2, lambda}
p_1=\frac{\vert b\vert}{a+\vert b\vert},\qquad p_2=\frac{a}{a+\vert b \vert},\qquad \lambda=\frac{a+\vert b\vert}{a\vert b\vert}\sqrt{\frac1{2\pi}} E L_1.
\end{equation}

Since the sets $\{R>a^2x^{-2}\}\cap \mathcal{U}_\delta^+$ and $\{R>b^2x^{-2}\}\cap \mathcal{U}_\delta^-$ used in the definition of $N_1(x)$ and $N_2(x)$ are disjoint, $N_1$ and $N_2$ are independent Poisson processes with the above intensities.

It follows (see for example \cite{Shreve}, Theorem 11.3.3) that the process $Q$ defined by $Q(0)=0$ and
\begin{equation}
Q(x)=(+1)\cdot N_1(x)+(-1)\cdot N_2(x),\qquad x>0,
\end{equation}
is a compound Poisson process, with jump sizes $\pm1$. Moreover, if $\widetilde{U}_1,\widetilde{U}_2,\ldots$ represent the successive jumps of $Q$, and if $N(x)=N_1(x)+N_2(x)$ represents the total number of jumps of $(Q(y))_{0\leq y\leq x}$, then $Q$ can be written in the form
\begin{equation}\label{alternate representation of Q}
Q(x)=\sum_{j=1}^{N(x)} \widetilde{U}_j, \qquad x>0,
\end{equation}
where $N$ is a Poisson process with intensity $\lambda$ and $\widetilde{U}_1, \widetilde{U}_2,\ldots$ is an i.i.d sequence of random variables, taking the value $+1$ with probability $p_1$ and the value $-1$ with probability $p_2$, given (\ref{p_1, p_2, lambda}) above.

Recall now the discussion following (\ref{correspondence between excursions}): the $j^\text{th}$ longest excursion $e^Y_{i,j}$ of $Y$ in $E_i$ is the absolute value of the $j^\text{th}$ longest scaled excursion of $W$ in $\widetilde{E}_{i}$, where the scaling map $s_{a,b}$ scales positive excursions by $a$, and negative excursions by $b$. This means that if $\widetilde{e}_{i,k}$ becomes the $j^\text{th}$ longest excursion of $W$ in $\widetilde{E}_i$ after scaling, then $$e^Y_{i,j}(u)=\left\vert s_{a,b}\left(e^W_{i,k}(u)\right)\right\vert=\frac{1}{\big\vert \sigma_{a,b} {\big(e^W_{i,k}\big)}\big\vert} \left \vert e_{i,k}^W \left(\sigma_{a,b}^2\left(e_{i,k}^W\right) u\right)\right \vert,$$ and moreover by (\ref{representation no 2}) we also have 
that the sign $U_{i,j}$ of $X$ during the excursion $e^Y_{i,j}$  is the same as the sign $V_{i,k}$ of $W$ during the corresponding excursion $e^W_{i,k}$. Using this and the definitions of $N_1(x), N_2(x), N(x)$, and $Q(x)$ above, we obtain
\begin{eqnarray*}
Q(x)&=&N_1(x)-N_2(x)\\
&=&N_{L_1}^{\{R>a^2x^{-2}\}\cap \mathcal{U}^+}-N_{L_1}^{\{R>b^2x^{-2}\}\cap\mathcal{U}^-}\\
&=&N_{L_1}^{\{a^{-2} R>x^{-2}\}\cap\mathcal{U}^+}-N_{L_1}^{\{b^{-2}R>x^{-2}\}\cap\mathcal{U}^-}\\
&=&N_{L_1}^{\{R\circ s_{a,b}>x^{-2}\}\cap\mathcal{U}^+}-N_{L_1}^{\{R\circ s_{a,b}>x^{-2}\}\cap\mathcal{U}^-}\\
&=&\sum_{j=1}^{N(x)}U_{1,j},
\end{eqnarray*}
where the last equality expresses the fact that the sum of signs the first $N(x)$ longest excursions of $X$ is the same as the sum as the number $N_1(x)$ of scaled positive excursions of $W$ with length greater than $x^{-2}$ minus the number $N_2(x)$ of scaled negative excursions of $W$ with length greater than $x^{-2}$, which by the previous discussion is the same.

Since $(N(x))_{x>0}$ is a Poisson process, in particular it increases only by jumps of size $1$, and from its definition we have $N(\infty)=N_{L_1}^{(0,\infty)}=\infty$. Using this, and comparing the above representation of $Q(x)$ with (\ref{alternate representation of Q}), it follows that $U_{1,j}=\widetilde{U}_j$ for all $j\geq 1$, and therefore $(U_{1,j})_{j\geq 1}$ is an i.i.d. sequence of random variables taking the values $+1$ and $-1$ with probabilities $p_1=\frac{\vert b \vert}{a+\vert b\vert}$, respectively $p_2=\frac{a}{a+\vert b\vert}$, as needed.

Using the strong Markov property of $X$ and $Y$ at the stopping time $$\xi_1=\inf\{t>1: Y_t=0\}=\inf\{t>1:X_t=0\},$$ the argument above shows that the sequence $(U_{2,j})_{j\geq1}$ is an i.i.d. sequence with the same distribution as $(U_{1,j})_{j\geq1}$, and also independent of it. Inductively, using the strong Markov property at the stopping times $$\xi_{i+1}=\inf\{t>\xi_i+1: Y_t=0\}=\inf\{t>\xi_i+1:X_t=0\},$$ we obtain that $(U_{i+2,j})_{j\geq1}$ is an i.i.d. sequence of random variables with the same distribution as $(U_{i^\prime,j})_{1\leq i^\prime \leq i+1, j\geq 1}$, and also independent of it. This shows that $(U_{i,j})_{i,j\geq1}$ is an i.i.d. sequence of random variables with distribution given in the conclusion of the theorem.

To conclude the proof, we have left to show that the sequence $(U_{i,j})_{i,j\ge 1}$ is also independent of $Y$. We will first show that the process $\widetilde{W}$ defined by (\ref{definition of widetilde W}) below is a skew Brownian motion. Note that by (\ref{representation no 1}), the $\sigma$-algebra $\sigma(Y_s:s\leq t)$ is the same as the $\sigma$-algebra
\begin{equation*}
\sigma\left( \frac{1}{\sigma_{a,b}(X_s)} X_{s}:s\leq t\right)=\sigma\left( \left\vert \frac{1}{\left\vert\sigma_{a,b}(X_s)\right\vert} X_{s}\right\vert:s\leq t\right)=\sigma\left(\big\vert\widetilde{W}_s \big\vert:s\leq t\right) ,
\end{equation*}
so $\mathcal{F}^Y_t = \mathcal{F}_{t}^{\vert \widetilde{W} \vert}$, where $\widetilde{W}$ represents the process defined by
\begin{equation}\label{definition of widetilde W}
\widetilde{W}_t=\frac{1}{\left\vert \sigma_{a,b}(X_{t})\right\vert }X_{t}, \qquad t\geq 0.
\end{equation}

Applying the It\^{o}-Tanaka formula to the function $f(x)=\dfrac{1}{\vert\sigma_{a,b}(x)\vert} x$ (a difference of two convex functions) and to $X$, we obtain
\begin{eqnarray*}
\frac{1}{\left\vert \sigma_{a,b}(X_{t})\right\vert }X_{t}&=&\int_0^t \frac{1}{\left\vert \sigma_{a,b}(X_{s})\right\vert } d X_{s} + \frac12\left( \frac{1}{a}-\frac{1}{\vert b\vert}\right) L_t^0(X)\\
&=&\int_0^t \frac{1}{\left\vert \sigma_{a,b}(X_{s})\right\vert } \sigma_{a,b}(X_s) d B_{s} + \frac12\left( \frac{1}{a}-\frac{1}{\vert b\vert}\right) L_t^0(X)\\
&=&\int_0^t \mathrm{sgn}(X_s) d B_{s} + \frac12\left( \frac{1}{a}-\frac{1}{\vert b\vert}\right) L_t^0(X),
\end{eqnarray*}
where $L^0(X)$ represents the semimartingale local time of $X$ at the origin .

The above shows that $\widetilde{W}$ is a continuous semimartingale, and since
$$\big\langle \widetilde{W}\big\rangle_t=\int_0^t \mathrm{sgn}^2(X_s) d\langle B\rangle_s =\int_0^t 1 ds=t, \qquad t\ge 0,$$
the martingale part of $\widetilde{W}$ is a Brownian motion $\widetilde{B}$. The previous equation is thus equivalent to
\begin{equation}\label{representation of tilde W}
\widetilde{W}_t=\widetilde{B}_t+\frac12\left( \frac{1}{a}-\frac{1}{\vert b\vert}\right) L_t^0(X), \qquad t \ge 0.
\end{equation}

Note that since $\widetilde{W}$ is a continuous semimartingale with $\langle \widetilde{W}\rangle_t =t$, it spends zero Lebesgue measure at the origin (see Corollary VI.1.6 of \cite{Revuz and Yor}). It follows that the symmetric semimartingale local time $\widehat{L}^0\big(\widetilde{W}\big)$ of $\widetilde{W}$ at the origin is given by
\begin{eqnarray*}
\widehat{L}_t^0\big(\widetilde{W}\big)&=&\lim_{\epsilon \searrow 0} \frac{1}{2\varepsilon}\int_0^t 1_{(-\varepsilon,\varepsilon)}\big(\widetilde{W}_s\big) d \big\langle \widetilde{W}\big\rangle_s\\
&=&\frac12 \left( \lim_{\varepsilon \searrow 0} \frac{1}{\varepsilon}\int_0^t 1_{[0,\varepsilon)}\big(\widetilde{W}_s\big) ds +\lim_{\varepsilon \searrow 0} \frac{1}{\varepsilon}\int_0^t 1_{(-\varepsilon,0]}\big(\widetilde{W}_s\big) ds\right)\\
&=&\frac12 \left( \lim_{\varepsilon \searrow 0} \frac{1}{\varepsilon}\int_0^t 1_{[0,\varepsilon)}\left(a^{-1} X_s\right) ds +\lim_{\varepsilon \searrow 0} \frac{1}{\varepsilon}\int_0^t 1_{(-\varepsilon,0]}\big({\vert b\vert}^{-1} X_s\big) ds\right)\\
&=&\frac12 \left(\lim_{\varepsilon \searrow 0} \frac{1}{\varepsilon}\int_0^t 1_{[0,a\varepsilon)}\left(X_s\right) \frac1{a^2}d\langle X \rangle_s +\lim_{\varepsilon \searrow 0} \frac{1}{\varepsilon}\int_0^t 1_{(-\vert b\vert\varepsilon,0]}\left( X_s\right) \frac{1}{b^2}d\langle X \rangle_s\right)\\
&=&\frac12\left(\frac1a{L}_t^0(X)+\frac{1}{\vert b \vert} {L}_t^0(-X)\right),
\end{eqnarray*}
where in the penultimate equality we have used the fact that the quadratic variation of $X$ is given by $\langle X\rangle_t=\int_0^t \sigma_{a,b}^2 (X_s) ds$, $t\ge 0$.

Since $X$ is a continuous local martingale which spends zero Lebesgue time at the origin, by Corollary VI.1.9 of \cite{Revuz and Yor} it follows that $L_t^0(X)=L_t^0(-X)$ for any $t\ge 0$, and therefore we obtain $\widehat{L}_t^0\big(\widetilde{W}\big)=\frac12\left (\frac1a+\frac{1}{\vert b \vert} \right){L}_t^0(X)$. It follows that the equation (\ref{representation of tilde W}) can rewritten in the equivalent form
\begin{equation}\label{definition of W tilde}
\widetilde{W}_t=\widetilde{B}_t+\frac{\vert b\vert-a}{\vert b\vert+a} \widehat{L}^0_t\big(\widetilde{W}\big), \qquad t\ge 0,
\end{equation}
which shows that $\widetilde{W}$ is a skew Brownian motion with parameter $\alpha = \frac12\left(1+\frac{\vert b \vert-a}{\vert b \vert+a}\right)=\frac{\vert b \vert}{a+\vert b \vert}$ (see for example p. 401 of \cite{Revuz and Yor}, or \cite{Harrison-Shepp} and the references cited therein for more details about this process).

It is known that the skew Brownian motion with parameter $\alpha \in(0,1)$ is a Markov process which behaves like ordinary Brownian motion on each of the two intervals $(-\infty,0)$ and $(0,\infty)$, and that starting at the origin it enters the positive real axis with probability $\alpha$. Further, the explicit transition densities of the process are known (see for example \cite{Revuz and Yor}, pp. 82).

\label{pagina de referinta start descompuere Ito-McKean}
We are now ready to show that the sequence $(U_{i,j})_{i,j\ge 1}$ is also independent of $\mathcal{F}^Y_\infty$. In \cite{Ito}, It\^{o} and McKean gave a decomposition of the excursion process of Brownian motion as follows. If $e^W$ is the excursion process of the Brownian motion $W$, let $(I_n)_{n\ge 1}$ be a certain numbering of the excursion intervals which depends only on the zero set $Z^W=\{t\ge0:W_t=0\}$ of $W$ (see the Appendix \ref{Appendix} for the actual choice of this numbering), consider the corresponding \emph{unscaled normalized excursion} processes $(\mathbf{e}_n)_{n\ge 1}$ of $W$ defined by $\mathbf{e}_n(t)=\frac{1}{\sqrt{\vert I_n\vert}} \left\vert W_{t\vert I_n\vert +\inf I_n}\right\vert$, $0\le t\le 1$, $n=1,2,\ldots$, and the corresponding signs $(U_n)_{n\ge 1}$ of excursions of $\widetilde{W}$ defined by $U_n=\mathrm{sgn}(W_t)$ if $t\in I_n$, $n=1,2,\ldots$. It\^{o} and McKean showed that $(\mathbf{e}_n)_{n\ge 1}$ are equal in law and mutually independent, $(U_n)_{n\ge 1}$ is an i.i.d. sequence of random variables which take the values $\pm1$ with probability $\frac12$, and that $(\mathbf{e}_n)_{n\ge 1}$, $(U_n)_{n\ge 1}$ and $Z^W$ are independent. Blumenthal (\cite{Blumenthal}, pp. 113 -- 118) showed that the actual numbering of the excursion intervals in unimportant, as long as it is a measurable function of the $\sigma$-algebra $\mathcal{Z}^W=\sigma(l=l(\omega),r=r(\omega):(l,r) \text{ excursion interval of } W)$ generated by the random variables which represent the left and right endpoints of the excursion intervals of $W$. He showed (Theorem 1.5 in \cite{Blumenthal}) that if $L_1,\ldots,L_n$ are $\mathcal{Z}^W$-measurable random variables which represent distinct left ends of excursion intervals of $W$, and if $\mathbf{e}_{L_i}(t)=\frac{1}{\sqrt{R_i-L_i}}\left\vert W_{L_i+t(R_i-L_i)}\right\vert$, $0\le t\le 1$, where $R_i$ is the right endpoint of the excursion interval with left endpoint $L_i$, and if  $U_{L_i}=\mathrm{sgn}(W_t)$ for $t\in(L_i,R_i)$, $i=1,\ldots,n$, then for each $i=1,\ldots,n$, $\textbf{e}_{L_i}$ has the law of an unsigned normalized Brownian excursion and $P(U_{L_i}=1)=1-P(U_{L_i}=-1)=\frac12$, and
\[
\textbf{e}_{L_1},\ldots,\textbf{e}_{L_n}, U_{L_1},\ldots,U_{L_n},\mathcal{Z}^W
\]
are independent.

Since the absolute value of a skew Brownian motion is a reflecting Brownian motion, and skew Brownian motion with parameter $\alpha\in [0,1]$ starting at the origin enters the positive axis with probability $\alpha$ (skew Brownian motion with parameter $\alpha=\frac12$ is the same as ordinary Brownian motion), the same proof as in \cite{Blumenthal} (see the proof of Theorem 5.1, pp. 115 -- 118) shows that the result above is true if we replace the Brownian motion $W$ by a skew Brownian motion $W^{\alpha}$ with parameter $\alpha\in[0,1]$, and the constant $\frac12$ which gives the distribution of signs of Brownian motion by the constant $\alpha$ which gives the distribution of signs of skew Brownian motion with parameter $\alpha$. We are now going to apply this result to the process $\widetilde{W}$ defined in (\ref{definition of widetilde W}), which by (\ref{definition of W tilde}) is seen to be skew Brownian motion $\widetilde{W}$ with parameter $\alpha=\frac{\vert b \vert}{a+\vert b \vert}$.
\label{pagina de referinta end descompuere Ito-McKean}

Consider now the numbering $(\widetilde{I}_{n,k})_{n,k\ge 1}$, introduced on pp. \pageref{sign choice} -- \pageref{iid sign choice}, of the excursion intervals of the skew Brownian motion $\widetilde{W}$, and let $L_{n,k}, R_{n,k}$ be the corres\-ponding left and right endpoints of these intervals, that is $\widetilde{I}_{n,k}=(L_{n,k},R_{n.k})$, $n,k\ge 1$. The numbering depends only on the zero set of $Z^{\widetilde{W}}$ of $\widetilde{W}$, so in particular $(L_{n,k})_{n,k\ge 1}$ are distinct $\mathcal{Z}^{\widetilde{W}}$-measurable random variables. If $\widetilde{\textbf{e}}_{L_{n,k}}(t)=\frac{1}{\sqrt{\vert\widetilde{I}_{n,k}\vert}}\left\vert \widetilde{W}_{L_{n,k}+t\vert\widetilde{I}_{n,k}\vert}\right\vert$, $0\le t\le 1$, $n,k\ge 1$, are the unsigned normalized excursion of $\widetilde{W}$ and $\widetilde{U}_{L_{n,k}}=\mathrm{sgn}(\widetilde{W}_t)$ for $t\in\widetilde{I}_{n,k}$, $n,k\ge1$, are the corresponding signs of excursions of $\widetilde{W}$, by the previous remark it follows that the collection
\[
\left(\widetilde{\textbf{e}}_{L_{n,k}}\right)_{n,k\ge 1},\left( \widetilde{U}_{L_{n,k}}\right)_{n,k\ge 1},\mathcal{Z}^{\widetilde{W}}
\]
is independent.

Note that $\big\vert \widetilde{W}_t\big\vert=\sqrt{\vert \widetilde{I}_{n,k}\vert} \, \widetilde{\mathbf{e}}_{L_{n,k}}\big( t\slash {\vert \widetilde{I}_{n,k}\vert}\big)$ if $t\in \widetilde{I}_{L_{n,k}}$ and $\widetilde{W}_t=0$ otherwise, so $\widetilde{W}_t$ is a $\sigma(\widetilde{\textbf{e}}_{L_{n,k}}, n,k\ge 1)\vee  \mathcal{Z}^{\widetilde{W}}$-measurable random variable for each $t\ge 0$. It follows that $\mathcal{F}^{\vert \widetilde{W}\vert}_\infty \subset\sigma(\widetilde{\textbf{e}}_{L_{n,k}}: n,k\ge 1)\vee  \mathcal{Z}^{\widetilde{W}}$, and therefore the sequence $\big( \widetilde{U}_{L_{n,k}}\big)_{n,k\ge 1}$ is independent of $\mathcal{F}^{\vert \widetilde{W}\vert}_\infty$. Recalling the definition (\ref{definition of widetilde W}) of $\widetilde{W}$ and (\ref{representation no 1}), we see that $\vert \widetilde{W}\vert=Y$ and $\mathrm{sgn}(X_t)=\mathrm{sgn}(\widetilde{W}_t)$. It follows that the numbering of the excursion intervals of $\widetilde{W}$ and $Y$ is the same, $\widetilde{I}_{n,k}=I_{n,k}$, $n,k\ge 1$. Since by definition $\widetilde{U}_{L_{n,k}}$ is the sign of $\widetilde{W}$ during the interval $\widetilde{I}_{L_{n,k}}$, and $U_{n,k}$ is the sign of $X$ during the excursion interval $I_{n,k}$, it follows that $\widetilde{U}_{L_{n,k}}=U_{n,k}$ for all $n,k\ge 1$. We conclude that $(U_{n,k})_{n,k\ge 1}=(\widetilde{U}_{L_{n,k}})_{n,k\ge 1}$ is independent of $\mathcal{F}^{\vert \widetilde{W}\vert}_\infty=\mathcal{F}^{Y}_\infty$. This, together with the previous part of the proof shows that $U$ is an i.i.d.  sign choice for $Y$ in the sense of Definition \ref{iid sign choice}, concluding the proof of the theorem.
\end{proof}

\section{Extensions\label{Section 2}}

The result of the previous section can be extended as follows.

\begin{theorem}
\label{theorem 2}Let $\sigma $ be a measurable function on $\mathbb{R}$ such
that $\left\vert \sigma \right\vert $ is bounded above and below by positive
constants and suppose there exists a strictly increasing function $f$ on $%
\mathbb{R}$ such that%
\begin{equation}
\left\vert \sigma \left( x\right) -\sigma \left( y\right) \right\vert
^{2}\leq \left\vert f\left( x\right) -f\left( y\right) \right\vert ,\qquad
x,y\in \mathbb{R}.  \label{condition on sigma}
\end{equation}

Further, assume that $\sigma $ is an odd function on $\mathbb{R}^{\ast }$
and $x\sigma \left( x\right) \geq 0$ for all $x\in \mathbb{R}$.

The following assertions are true relative to the stochastic differential equation
\begin{equation}\label{SDE3}
X_{t}=\int_{0}^{t}\sigma \left( X_{s}\right) dB_{s},\qquad t\geq 0,
\end{equation}
where $B$ is a $1$-dimensional Brownian motion starting the origin.
\begin{enumerate}
\item (Existence) There exists a weak solution $(X,B,\mathbb{F})$, and every such solution is also a $\vert x\vert$-strong solution.

\item (Uniqueness) $\vert x\vert$-strong uniqueness holds.

\item (Construction) A weak solution and a $\vert x \vert $-strong solution can be constructed from a given Brownian motion $(B,\mathbb{F}^B)$ as follows: $\big(  U Y,B,\mathbb{G}\big) $, where $Y$ is the pathwise unique solution to
\begin{equation}\label{equation for Y}
Y_{t}=\int_{0}^{t}\left\vert \sigma \left( Y_{s}\right) \right\vert
dB_{s}+\widehat{L}_{t}^{0}\left( Y\right) ,\qquad t\geq 0,
\end{equation}
$\widehat{L}^{0}\left( Y\right)$ is the symmetric semimartingale local time of $Y$, $U$ is an i.i.d. sign choice for $Y$ taking the values $\pm1$ with equal probability, and $\mathbb{G}$ is certain enlargement of the filtration $\mathbb{F}^B$ which satisfies the usual hypothesis (see Appendix \ref{Appendix}).

\item (Representation) Every weak solution $(X,B,\mathbb{F})$ has the representation $X=U Y$, where $U$ and $Y$ are as above.
\end{enumerate}
\end{theorem}

\begin{proof}
The fact that (\ref{SDE3}) has a weak solution follows immediately from the classical results of Engelbert and Schmidt (\cite{Engelbert-Schmidt}): since the set $I(\sigma)$ of non-local integrability of $\sigma^{-2}$ is in this case empty, it coincides with the set $Z(\sigma)=\varnothing$ of zeroes of $\sigma$, so (\ref{SDE3}) has a weak solution, and moreover the solution is also unique in the sense of probability law.

If $\left( X,B,\mathbb{F}\right) $ is a weak solution of (\ref%
{SDE3}), $d \langle X \rangle_t=\sigma^2(X_t) dt$ is absolutely continuous with respect to $dt$ (recall that by hypothesis $\sigma$ is bounded), and therefore $X$ spends zero Lebesgue time at the origin (see Corollary VI.1.6 in \cite{Revuz and Yor}). Applying the It\^{o}-Tanaka formula we obtain%
\begin{eqnarray}
\left\vert X_{t}\right\vert &=&\int_{0}^{t}\big(1_{(0,\infty)}\left( X_{s}\right)-1_{(-\infty,0]}(X_s)\big)dX_s+L_{t}^{0}\left( X\right) \notag\label{equation for |X|}\\
&=&\int_{0}^{t}\big(1_{[0,\infty)}\left( X_{s}\right)-1_{(-\infty,0)}(X_s)\big)dX_s+L_{t}^{0}\left( X\right) \notag\\
&=&\int_{0}^{t}\mathrm{sgn}(X_s)\sigma(X_s) dB_s+L_{t}^{0}\left( X\right),\notag
\end{eqnarray}%
where $\mathrm{sgn}(x)=1_{[0,\infty)}(x)-1_{(-\infty,0)}(x)$ and $L_t^0(X)=\lim_{\varepsilon\rightarrow0}\frac1\varepsilon\int_0^t 1_{[0,\varepsilon)}(X_s)d\langle X\rangle_s$ is the semimartingale local time of $X$ at the origin.

Denoting by $\widehat{L}^0_t(\vert X\vert)=\lim_{\varepsilon\rightarrow0}\frac1{2\varepsilon}\int_0^t 1_{[0,\varepsilon)}(\vert X_s\vert )d\langle \vert X\vert \rangle_s$ the symmetric semimartingale local time of $\vert X\vert$ at the origin, we have $$\widehat{L}^0_t(\vert X\vert)=\lim_{\varepsilon\rightarrow0}\frac1{2\varepsilon}\int_0^t 1_{(-\varepsilon,\varepsilon)}(X_s)d\langle X \rangle_s=\widehat{L}^0_t(X)=L^0_t(X),$$ since $X$ is a continuous local martingale, and therefore by Corollary VI.1.9 of \cite{Revuz and Yor} we have $\widehat{L}^0_t(X)=L_t^0 (X)$. Using the additional hypotheses on $\sigma $ and the fact that $X$ spends zero Lebesgue time at the origin, it follows that $(\vert X\vert,B,\mathbb{F})$ is a weak solution of the following stochastic differential equation %
\begin{equation}\label{aux3}
\left\vert X_{t}\right\vert =\int_{0}^{t}\left\vert \sigma \left( \left\vert
X_{s}\right\vert \right) \right\vert dB_{s}+\widehat{L}_{t}^{0}\left(\vert X\vert\right) ,\qquad
t\geq 0.
\end{equation}

The above SDE satisfies the conditions of Theorem 4.1 in \cite{Bass-Chen}
(with $a=\left\vert \sigma \right\vert $), so applying this theorem it
follows that (\ref{aux3}) has a strong solution and the solution is pathwise unique. By Theorem IX.1.7 in \cite{Revuz and Yor} it follows that every weak solution of (\ref{aux3}) is also a strong solution, so  in particular $\vert X\vert$ is $\mathbb{F}^B$-adapted and also pathwise unique. This shows that $(X,B,\mathbb{F})$ is a $\vert x\vert$-strong solution of (\ref{SDE3}) and that $\vert x\vert$-strong uniqueness holds for (\ref{SDE3}).

If $(Y,B,\mathbb{F}^B)$ is the unique strong solution of (\ref{equation for Y}) (note that this is the same SDE as (\ref{aux3}) above), by Theorem 3.3 in \cite{Bass-Chen} $Y$ is a non-negative process. If $U$ is an i.i.d. sign choice for $Y$ in the sense of Definition \ref{iid sign choice}, which takes the values $\pm 1$ with equal probability, we will show that ($UY,B,\mathbb{G}$) is a $\vert x\vert$-strong solution (and also a weak solution) of (\ref{SDE3}), where $\mathbb{G}$ is an enlargement of the filtration $\mathbb{F}^B$ we will indicate below. Note that since $\vert U Y\vert = Y$ and $Y$ is adapted to the filtration $\mathbb{F}^B$, in order to prove the claim it suffices to prove that ($UY,B,\mathbb{G}$) is a weak solution of (\ref{SDE3}).

To do this, we will first show that $B$ remains a Brownian motion under the enlarged filtration $\mathbb{G}$. Let $X=UY$ and note that the process $(A_t)_{t\ge0}$ defined by $A_t=\int_0^t\sigma^2(Y_s)ds$, $t\ge0$, is strictly increasing in $t$ (recall that by hypothesis $\vert \sigma\vert $ is bounded below by a positive constant), hence its inverse $(\alpha_t)_{t\ge0}$ defined by $\alpha_t=\inf\{s\ge0:A_s>t\}$, $t\ge 0$, is continuous in $t$. Set $\widetilde{X}_t=X_{\alpha_t}$, $\widetilde{Y}_t=Y_{\alpha_t}$, $\widetilde{U}_t=U_{\alpha_t}$, $\widetilde{B}_t=\int_0^{\alpha_t}\sigma(Y_s)dB_s$, and $\widetilde{\mathcal{G}}_t=\mathcal{G}_{\alpha_t}$, $t \ge 0$. By the time-change for martingales theorem (see \cite{Karatzas-Shreve}, Theorem 3.4.6) it follows that $(\widetilde{B}, \widetilde{\mathbb{G}})$ is a Brownian motion (although we have not specified the filtration $\mathbb{G}$ yet, all we are using here is that $\mathbb{G}$ is an enlarged filtration of $\mathbb{F}^B$ which satisfies the usual conditions, so $\int_0^{\cdot}\sigma(Y_s)dB_s$ is a $\mathbb{G}$-adapted process).

Using the substitution $A_s=u$ we obtain $$\widehat{L}^0_{\alpha_t}(Y)=\lim_{\varepsilon\rightarrow 0} \frac{1}{2\varepsilon}\int_0^{\alpha_t} 1_{(-\varepsilon,\varepsilon)}(Y_{s}) \sigma^2(Y_s) ds= \lim_{\varepsilon\rightarrow 0} \frac{1}{\varepsilon} \int_0^{t} 1_{(-\varepsilon,\varepsilon)}(\widetilde{Y}_u)du=\widehat{L}_t^0(\widetilde{Y}),$$ so from equation (\ref{equation for Y}) we obtain $\widetilde{Y}_t=\widetilde{B}_t+\widehat{L}_t^0(\widetilde{Y})$, which shows that $\widetilde{Y}$ is the reflecting Brownian motion on $[0,\infty)$ with driving Brownian motion with $\widetilde{B}$.

Next, we will show that $\widetilde{U}$ is an i.i.d. sign choice for $\widetilde{Y}$. To see this, note that since the time change $(\alpha_{t})_{t\ge0}$ is continuous, the continuity of $\widetilde{U}\widetilde{Y}$ follows from that of $UY$, so $\widetilde{U}$ is a sign choice.

\label{pagina de referinta start corespondenta A}
To prove that is also an i.i.d. sign choice, consider the interval excursions $(I_{i,j})_{i,j\ge1}$ and $(\widetilde{I}_{i,j})_{i,j\ge1}$ of $Y$, respectively $\widetilde{Y}$, as introduced on pp. \pageref{sign choice} -- \pageref{iid sign choice}. Since $\widetilde{Y}_{A}=Y_t$, we have $Y_t=0$ iff $\widetilde{Y}_{A_t}=0$, so $A$ establishes a (monotone) bijection between the zero set $Z^Y$ of $Y$ and the zero set $Z^{\widetilde{Y}}$ of $Y$, the bijection being $t\in Z^Y\mapsto A_t\in Z^{\widetilde{Y}}$. Since $A$ is an increasing process, for each interval $I_{i,j}=(g_{i,j},d_{i,j})$ there exists a unique interval $\widetilde{I}_{k,l}=(\widetilde{g}_{k,l},\widetilde{d_{k,l}})$ such that $A_{g_{i,j}}=\widetilde{g}_{k,l}$ and $A_{d_{i,j}}=\widetilde{d}_{k,l}$, and by an abuse of notation we write in this case $A({I_{i,j}})=\widetilde{I}_{k,l}$. Defined this way, $A(\cdot)$ is a bijection between the ordered excursion intervals of $Y$ and $\widetilde{Y}$. If $U_{i,j}$ and $\widetilde{U}_{k,l}$ are the restrictions of $U$ to $I_{i,j}$, respectively of $\widetilde{U}$ to $\widetilde{I}_{k,l}$, it follows that $\widetilde{U}_{k,l}=U_{i,j}$ if $A(I_{i,j})=\widetilde{I}_{k,l}$, and we obtain
\begin{eqnarray*}
E\left(\widetilde{U}_{k,l} \vert \mathcal{F}_{\infty}^{\widetilde Y}\right)&=&\sum_{i,j\ge 1}E\left(1_{\{A(I_{i,j})=\widetilde{I}_{k,l}\}}{U}_{i,j}\vert \mathcal{F}_{\infty}^{\widetilde{Y}}\right)\\
&=&\sum_{i,j\ge 1}1_{\{A(I_{i,j})=\widetilde{I}_{k,l}\}} E\left({U}_{i,j}\vert \mathcal{F}_{\infty}^{{Y}}\right)\\
&=&\sum_{i,j\ge 1}1_{\{A(I_{i,j})=\widetilde{I}_{k,l}\}} E\left({U}_{i,j}\right)\\
&=&E\left({U}_{i,j}\right)\\
&=&0,
\end{eqnarray*}
since $\{A(I_{i,j})=\widetilde{I}_{k,l}\}$ is an event in $\mathcal{F}^Y_{\infty}= \mathcal{F}^{\widetilde{Y}}_{\infty}$. This shows that the random variables $\widetilde{U}_{k,l}$, $k,l\ge 1$ are random variables independent of $\widetilde{Y}$, and identically distributed with $P(U_{k,l}=\pm1)=\frac12$. A similar argument can be used to prove that $(\widetilde{U}_{k,l})_{k,l\ge 1}$ is a sequence of independent random variables, also independent of $\widetilde{Y}$, so $\widetilde{U}$ is indeed a sign choice for $\widetilde{Y}$ in the sense of Definition \ref{iid sign choice}. Applying the construction in the Appendix \ref{Appendix} to $\widetilde{Y}$ and $\widetilde{U}$ with $\alpha=\frac12$ it follows that there exists an enlargement $\widetilde{\mathbb{G}}$ of the filtration $\mathbb{F}^{\widetilde{Y}}$, satisfying the usual hypotheses, for which $(\widetilde{U}\widetilde{Y},\widetilde{\mathbb{G}},P,P^b)$ is a Brownian motion, and also a strong Markov process. The filtration $\mathbb{G}$ mentioned in the statement of the theorem and earlier in the proof is the filtration $\mathbb{G}=(\mathcal{G}_t)_{t\ge 0}$ defined by $\mathcal{G}_t=\widetilde{\mathcal{G}}_{A_t}$, $t\ge 0$, and it is easy to see that it is an  enlargement of the filtration $\mathbb{F}^Y=\mathbb{F}^B$ and it satisfies the usual conditions.

Since $A_t=\inf\{s\ge0:\alpha_s>t\}=\inf\{s>0: \int_0^s \frac{1}{\sigma^2(\widetilde{Y}_u)}du>t\}$ is a $\widetilde{\mathbb{G}}$-stopping time (recall that $\widetilde{\mathbb{G}}$ is an enlargement of the filtration  $\mathbb{F}^{\widetilde{Y}}$), and it satisfies $A_{t+s}=A_s+A_{t}\circ \theta_{A_s}$, by the strong Markov property of $\widetilde{Y}$ we obtain $E\left(\widetilde{Y}_{A_{t+s}}\vert \widetilde{\mathcal{G}}_{A_s}\right)=E^{\widetilde{Y}_{A_s}}\left( \widetilde{Y}_{A_t}\right)$, or equivalent $E\left(Y_{t+s}\vert {\mathcal{G}}_{s}\right)=E^{{Y}_{s}}\left( {Y}_{t}\right)$. $(Y,\mathbb{G})$ is therefore a Markov process, and it can be shown that its symmetric local time $\widehat{L}^0(Y)$ is also a $\mathbb{G}$-adapted Markov process. Since $Y$ is a solution of (\ref{aux3}), it follows that $\int_0^t\frac{1}{\vert\sigma(Y_s)\vert} dY_s=B_t-B_0+\int_0^t\frac{1}{\vert\sigma(Y_s)\vert} d\widehat{L}^Y_s=B_t-B_0+\frac{1}{\vert\sigma(0)\vert} \widehat{L}^Y_t$, or equivalent $B_t=B_0+\int_0^t\frac{1}{\vert\sigma(Y_s)\vert} dY_s-\frac{1}{\vert\sigma(0)\vert} \widehat{L}^0_t (Y)$, so $B$ is also a $\mathbb{G}$-adapted Markov process. We obtain $E(B_{t+s}\vert \mathcal{G}_s)=E^{B_s} B_t=B_s$, since starting at $B_0=x$, the distribution of $B_t$ is normal with mean $x$ (recall that $(B,\mathbb{F}^B)$ is a Brownian motion). $(B,\mathbb{G})$ is therefore a continuous local martingale with quadratic variation at time $t$ equal to $t$, and by L\'{e}vy's characterization of Brownian motion it follows that $(B,\mathbb{G})$ is a Brownian motion, concluding the proof of the earlier claim.

In order to show that $X=UY$ satisfies the SDE (\ref{SDE3}), note that by the Definition \ref{iid sign choice} of the i.i.d. sign choice $U$ we have $\mathrm{sgn}\left( X_{t}\right) =U_t$ (when $Y_t=0$ we have $U_t=1$, which coincides with $\mathrm{sgn}(X_t)=\mathrm{sgn}(0)=1$). Using this, the hypothesis that $\sigma $ is an odd function on $R^{\ast}$, and the fact that $Y\geq 0$ ($Y$ is a time-changed reflecting Brownian motion on $[0,\infty )$), we obtain
\[
\sigma \left( X_{t}\right) =U_{t}\left\vert \sigma \left( Y_{t}\right)
\right\vert 1_{\mathbb{R}^{\ast }}\left( X_{t}\right) +\sigma \left(
0\right) 1_{\left\{ 0\right\} }\left( X_{t}\right) ,\qquad t\geq 0.
\]

Since $\left\vert \sigma \right\vert $ is bounded between positive
constants, the process $Y$ (hence $X$) spends zero Lebesgue time at the
origin, thus we have almost surely%
\[
\int_{0}^{t}\sigma \left( X_{s}\right)
dB_{s}=\int_{0}^{t}U_{s}\left\vert \sigma \left( Y_{s}\right) \right\vert
1_{\mathbb{R}^{\ast }}\left( X_{s}\right) dB_{s},
\]%
for any $t\geq 0$. Using the fact that $Y$ is a solution of (\ref{equation for Y}), we obtain%
\begin{eqnarray*}
\int_{0}^{t}\sigma \left( X_{s}\right) dB_{s} &=&\int_{0}^{t}U_{s}\left\vert
\sigma \left( Y_{s}\right) \right\vert 1_{\mathbb{R}^{\ast }}\left( X_{s}\right)
dB_{s} \\
&=&\int_{0}^{t}U_{s}1_{\mathbb{R}^{\ast }}\left( X_{s}\right)
dY_{s}-\int_{0}^{t}U_{s}1_{\mathbb{R}^{\ast }}\left( X_{s}\right) d\widehat{L}_{s}^{0}\left(
Y\right) \\
&=&\int_{0}^{t}U_{s}1_{\mathbb{R}^{\ast }}\left( X_{s}\right) dY_{s},
\end{eqnarray*}%
where the last equality follows from the fact that the local time $\widehat{L}_{s}^{0}\left(
Y\right) $ of $Y$ at the origin increases only when $Y_{s}$ (hence $%
X_{s} $) is at the origin.

Proceeding exactly as in the proof of Theorem \ref{theorem 1} in the case $a=-b=1$
(see pp. \pageref{pagina de start referinta} -- \pageref{pagina de sfarsit referinta}, and note that $\sigma _{a,b}(U_{s})=\sigma _{1,-1}(U_{s})=U_{s}$ in this case), we obtain%
\[
\int_{0}^{t}\sigma \left( X_{s}\right) dB_{s}=\int_{0}^{t}U_{s}1_{\mathbb{R}^{\ast }}\left( X_{s}\right) dY_{s}=U_{t}Y_{t}=X_{t}, \qquad t\ge 0.
\]%

This, together with the previous part shows that $(UY,B,\mathbb{G})$ is a weak solution and also a $\vert x\vert$-strong solution of (\ref{SDE3}), concluding the proof of the third claim of the theorem.

To prove the last claim, note that if $\left( X,B,\mathbb{F}\right) $ is a weak solution of (\ref{SDE3}), by the previous proof it follows that $Y=\left\vert X\right\vert$ is the pathwise unique strong solution of (\ref{equation for Y}). The process $X$ can thus be written in the form $X=U Y $, where $U=\mathrm{sgn}(X)$ is a $\mathbb{F}$-adapted process taking the values $\pm1$. Since $U Y=X$ is a continuous process,  it follows that $U$ is a sign choice for $Y$ in the sense of Definition \ref{sign choice}.

To prove that $U$ is an i.i.d. sign choice for $Y$, consider the ordered excursions intervals $(I_{i,j})_{i,j\geq 1}$ of $Y$ (as defined before Definition \ref{iid sign choice}),  and let $U_{i,j}$ be the restriction of $U$ to the interval $I_{i,j}$ (note that since $X=U Y$, the excursion intervals of $X$ and $Y$ are the same, so $U$ is constant during any such excursion interval). We have to show that $(U_{i,j})_{i,j\geq 1}$ is an i.i.d. sequence of random variables taking the values $\pm1$ with equal probability, and that it is also independent of $\mathbb{F}^Y$.

Since $(X,B,\mathbb{F})$ is a weak solution of (\ref{SDE3}) and $\vert \sigma\vert$ is by hypothesis bounded by positive constants, we have $lim_{t\rightarrow \infty} \langle X\rangle_t=\int_0^\infty \sigma^2(X_s) ds=\infty$, so by the time-change for martingales theorem (see \cite{Karatzas-Shreve}, Theorem 3.4.6) it follows that $X$ is a time-changed Brownian motion, more precisely $X_t=\widetilde{X}_{A_t}$, $t \ge 0$, where $(\widetilde{X},\widetilde{\mathbb{F}})$ is a Brownian motion, $A_t=\int_0^t \sigma^2(X_s) ds$, $t\ge 0$, and the filtration $\widetilde{\mathbb{F}}=(\widetilde{\mathcal{F}}_t)_{t\ge 0}$ is defined by $\widetilde{\mathcal{F}}_t={{\mathcal{F}}_{A_t^{-1}}}$. Also note that since $\sigma$ is an odd function on $\mathbb{R}^{\ast}$, we have $A_t=\int_0^t \sigma^2(\vert X_s\vert) ds$, which shows that the time change $(A_t)_{t\ge0}$ depends  only on the absolute value of $X$, and not on its sign.

Thanks to the symmetry of Brownian motion, $-\widetilde{X}$ and $\widetilde{X}$ have the same law, and since the time change $(A_t)_{t\ge 0}$ does not depend on the sign of $X$, the law of $X$ is the same as the law of $-X$. Alternatively, this can be seen from the fact that the solution of $(\ref{SDE3})$ is unique in the sense of probability law (Theorem 5.5.7 of \cite{Karatzas-Shreve}), and that $X$ and $-X$ are both solutions of (\ref{SDE3}).

Since $X$ is a time changed Brownian motion, it can be shown that its excursion process $(e_t^X)_{t>0}$ is a $(\mathcal{F}_{\tau_t})$-Poisson process (the proof is similar to that for standard Brownian motion, see pp. 448 -- 457 of \cite{Revuz and Yor}; see also Appendix \ref{Appendix} for the precise definitions and the construction of the excursion process in the case of Brownian motion). If $n$ is the characteristic measure of the excursion process $e^X$ of $X$, the fact that $X$ and $-X$ have the same law implies that the positive and negative excursions of $X$ are distributed the same under $n$, and in particular $n_+(R>x)=n_-(R>x)$, for any $x>0$, where $n_+$ and $n_-$ denote the restrictions of $n$ to the set of the positive excursions, respectively to the set of the negative excursions of $X$.

We can now proceed as in the corresponding part of the proof of Theorem \ref{theorem 1} in the case $a=-b=1$ (see p. \pageref{N_1 and N_2}). Consider $Q(x)=N_1(x)-N_2(x)$, where $N_1(x) = N_{L_1}^{(x^{-2},\infty)\cap\mathcal{U}_\delta^+}$ and $N_2(x)=N_{L_1}^{(x^{-2},\infty)\cap\mathcal{U}_\delta^-}$ represent the number of positive, respectively negative excursions of $X$ starting before time $1$, with lengths greater than $x^{-2}$. A similar proof shows that $N_1(x)$ and $N_2(x)$ are independent Poisson processes with the same intensity $\lambda=n_+(R>x^{-2})=n_-(R>x^{-2})$, so $Q(x)$ is a compound Poisson process with jump sizes $\pm1$. Also, if $N(x)=N_1(x)+N_2(x)$ represents the total number of jumps of $(Q(y))_{0\leq y \leq x}$ and $\widetilde{U}_1, \widetilde{U}_2,\ldots$ are the successive jumps of $Q$, then
\begin{equation}
Q(x)=\sum_{j=1}^{N(x)} \widetilde{U}_j, \qquad x>0,
\end{equation}
where $N(x)$ is a Poisson process with intensity $2\lambda$ and  $\big(\widetilde{U}_j\big)_{j\geq1}$ is an i.i.d. sequence of random variables taking the values $\pm1$ with probability $\frac12$.

Recall that we defined $U_{i,j}$ as the sign of $X$ during the $j^\text{th}$ longest excursion of $Y$ starting before time $1$. Since $X=U Y$, the interval excursions of $X$ and $Y$ are the same. Comparing with the above, and since $N(\infty)=N_{L_1}^{(0,\infty)}=\infty$, we conclude that $U_{1,j}=\widetilde{U}_j$ for all $j\geq 1$, so $(U_{1,j})_{j\geq1}$ is an i.i.d. sequence of random variables taking the values $\pm 1$ with equal probability. The fact that the whole sequence $(U_{i,j})_{i,j \geq 1}$ has the same properties being similar to that in the proof of Theorem \ref{theorem 1}, we omit it.

To conclude the proof, we have left to show that the sequence $(U_{i,j})_{i,j \geq 1}$ is also independent of $Y$.

We have seen that $X$ is a time change of a Brownian motion $\widetilde{X}$, more precisely $X_t=\widetilde{X}_{A_t}$, where $A_t=\int_0^t \sigma^2(\vert X_s\vert) ds$. In the proof of Theorem \ref{theorem 1} (p. \ref{pagina de referinta start descompuere Ito-McKean}) we showed that the sequence of signs of excursions of a skew Brownian motion with parameter $\alpha\in [0,1]$ is independent on its absolute value. Applying this to $\widetilde{X}$ (Brownian motion, which is the same as a skew Brownian motion with parameter $\alpha=\frac12$), it follows that the sequence $(\widetilde{V}_{i,j})_{i,j\ge 1}$ is independent of $\vert \widetilde{X} \vert$, where $\widetilde{V}_{i,j}$ is the sign on $\widetilde{X}$ during the excursion interval $\widetilde{I}_{i,j}$, and $(\widetilde{I}_{i,j})_{i,j\ge 1}$ are the ordered excursion intervals of $\widetilde{X}$, as defined on pp. \pageref{sign choice} - \pageref{iid sign choice}. Recall that $U_{i,j}$ is the sign of $U$ (which is the same as the sign of $X$) during the excursion interval $I_{i,j}$ of $Y$. From the double representation of $X_t=U_t Y_t=\widetilde{X}_{A_t}$, where $A_t=\int_0^t \sigma^2(\vert X_s\vert) ds=\int_0^t \sigma^2(Y_s) ds$, $t\ge 0$, exactly as in the proof on p. \ref{pagina de referinta start corespondenta A}, it can be seen that there is a bijective correspondence  between the excursion intervals of $\widetilde{X}$ and of $Y$, defined by $A(I_{i,j})=\widetilde{I}_{k,l}$ if $I_{i,j}=(g_{i,j},d_{i,j})$, $\widetilde{I}_{k,l}=(\widetilde{g}_{k,l}, \widetilde{d}_{k,l})$, $A_{g_{i,j}}=\widetilde{g}_{k,l}$, and $A_{d_{i,j}}=\widetilde{g}_{k,l}$. Note that if $A(I_{i,j})=\widetilde{I}_{k,l}$, then $U_{i,j}=\widetilde{V}_{k,l}$, since $U_{i,j}=\mathrm{sgn}(X_t)$ for $t\in I_{i,j}$ and $\widetilde{V}_{k,l}=\mathrm{sgn}(\widetilde{X}_t)$ for $t\in \widetilde{I}_{k,l}$, which is the same as $\mathrm{sgn}(\widetilde{X}_{A_t})=\mathrm{sgn}(X_t)$ for $t \in I_{i,j}$. We obtain therefore the following representation
\begin{equation}\label{representation of U_i,j}
U_{i,j}=\sum_{k,l\ge 1}\widetilde{V}_{k,l}1_{\{A(I_{i,j})=\widetilde{I}_{k,l}\}}, \qquad i,j\ge 1.
\end{equation}

We pause for a moment to show that $\mathcal{F}^{\vert\widetilde{X}\vert}_\infty=\mathcal{F}^Y_\infty$. Since $A_t^{-1}=\inf\{s\ge0:\int_0^t \sigma^2(Y_s) ds>t\}$ is a $\mathcal{F}^Y_\infty$- measurable random variable, the representation $\vert\widetilde{X}_t\vert=Y_{A_t^{-1}}$ shows that $\mathcal{F}^{\vert\widetilde{X}\vert}_\infty\subset\mathcal{F}^Y_\infty$. Conversely, noting that $A_t=\int_0^t\frac{1}{\sigma^2(\vert\widetilde{X}_s\vert)} ds$ is $\mathcal{F}^{\vert\widetilde{X}\vert}_\infty$-measurable, the representation $Y_t=\vert\widetilde{X}_{A_t}\vert$ shows that we also have $\mathcal{F}^Y_\infty\subset\mathcal{F}^{\widetilde{X}}_\infty$, so $\mathcal{F}^{\widetilde{X}}_\infty=\mathcal{F}^Y_\infty$ as claimed.

Using the independence of $(\widetilde{V}_{i,j})_{i,j\ge 1}$ and $\vert\widetilde{X}\vert$, the fact that the events $\{A(I_{i,j})=\widetilde{I}_{k,l}\}$ are in $\mathcal{F}_\infty^Y=\mathcal{F}^{\vert\widetilde{X}\vert}_{\infty}$ for all $i,j,k,l\ge 1$, and (\ref{representation of U_i,j}), we obtain
\begin{eqnarray*}
&&P\left(U_{i_m,j_m}=u_m, m=\overline{1,n}\big\vert \mathcal{F}_{\infty}^{Y}\right)=\\
&=&\sum_{k_m,l_m\ge 1, m=\overline{1,n}}P\left(\widetilde{V}_{k_m,l_m}=u_m, A(I_{i_m,j_m})=\widetilde{I}_{k_m,l_m}, m=\overline{1,n} \big\vert \mathcal{F}_{\infty}^{{Y}}\right)\\
&=&\sum_{k_m,l_m\ge 1, m=\overline{1,n}}1_{\{A(I_{i_m,j_m})=\widetilde{I}_{k_m,l_m}, m=\overline{1,n}\}} P\left(\widetilde{V}_{k_m,l_m}=u_m, m=\overline{1,n} \big\vert \mathcal{F}_{\infty}^{\vert\widetilde{X}\vert}\right)\\
&=&\sum_{k_m,l_m\ge 1, m=\overline{1,n}}1_{\{A(I_{i_m,j_m})=\widetilde{I}_{k_m,l_m}, m=\overline{1,n}\}} P\left(\widetilde{V}_{k_m,l_m}=u_m, m=\overline{1,n}\right)\\
&=&P\left(U_{i_m,j_m}=u_m, m=\overline{1,n}\right),
\end{eqnarray*}
for any $n\ge 1$, $u_1,\ldots,u_n\in\{-1,1\}$, and distinct $(i_1,j_1),\ldots,(i_n,j_n)$, which shows that the sequence $(U_{i,j})_{i,j\ge 1}$ is independent of $Y$. This, together with the previous part shows that $U$ is an i.i.d. sign choice for $Y$, concluding the proof of the theorem.
\end{proof}

We conclude with some remarks on the possibility of extending the previous
theorem by removing the two additional hypotheses: $\sigma $ is an odd
function on $\mathbb{R}^{\ast }$ and $x\sigma \left( x\right) \geq 0$ for
all $x\in \mathbb{R}$.

If $\sigma $ satisfies Le Gall's condition (\ref{Nakao condition}), then $%
\sigma $ has countably many discontinuities. If the set of discontinuities
do not have a limit point, then one can use a stopping time argument to
reduce the problem to the case when $\sigma $ has just one jump discontinuity, which,
without loss of generality may be assumed to be the origin. Either $\sigma $ has
the same sign on both sides of the origin (in which case the result of Le Gall
applies), or $\sigma $ has different signs on the left and on the right of
the origin. This shows that the additional hypothesis $x\sigma \left( x\right)
\geq 0$ for $x\in \mathbb{R}$ is not essential.

The condition that $\sigma $ is an odd function on $\mathbb{R}^{\ast }$ (if $%
\sigma \left( 0\right) =0$ then the solution of (\ref{SDE3}) is not even
weakly unique, so we avoided this case) does not seem to be essential for
the result, but it is a key element of the proof. It may be possible to
remove this hypothesis, using the following idea.

If $(X,B,\mathbb{F})$ is a weak solution of (\ref{SDE3}), from It\^{o}-Tanaka formula we obtain%
\begin{equation}\label{aux4}
\left\vert X_{t}\right\vert =\int_{0}^{t}\left\vert \sigma \left(
X_{s}\right) \right\vert dB_{s}+\widehat{L}_{t}^{0}\left( \vert X\vert \right) ,\qquad t\geq 0,
\end{equation}
where $\widehat{L}^0(X)$ is the symmetric local time of $X$ at the origin.

If $\sigma$ is a measurable function on $\mathbb{R}$ which satisfies condition (\ref{condition on sigma}), and $\vert\sigma\vert$ is bounded by positive constants, by Theorem 4.1 in \cite{Bass-Chen} the SDE
\begin{equation}\label{aux5}
Y_{t} =\int_{0}^{t}\left\vert \sigma \left(
Y_{s}\right) \right\vert dB_{s}+\widehat{L}_{t}^{0}\left( Y \right) ,\qquad t\geq 0,
\end{equation}
has a continuous strong solution which is pathwise unique, and the same is true for the SDE
\begin{equation}\label{aux6}
Y_{t} =-\int_{0}^{t}\left\vert \sigma \left(
Y_{s}\right) \right\vert dB_{s}-\widehat{L}_{t}^{0}\left( Y \right) ,\qquad t\geq 0.
\end{equation}

If $Y^1$ and $Y^2$ denote the corresponding (pathwise unique) solutions of (\ref{aux5}), respectively (\ref{aux6}), from Theorem 3.3 in \cite{Bass-Chen} it follows that $Y^1\ge 0\ge Y^2$ a.s. Comparing (\ref{aux4}) with (\ref{aux5}) and (\ref{aux6}), it follows that (\ref{aux4}) has a pathwise unique non-negative solution ($X=Y^1\ge0$), and a
pathwise unique non-positive solution ($X=Y^{2}\le0$). The ``generic'' solution of (\ref{aux4}) (and hence of (\ref{SDE3}%
)) should therefore be a ``mixture'' of $Y^1$ and $Y^2$, constructed by choosing with certain probabilities (depending on $\sigma(0+)$ and $\sigma(0-)$, as in Theorem \ref{theorem 1} and Theorem \ref{theorem 2}) between the excursions away from zero of $Y^1$ and $Y^2$, and patching them together.


However, we were not able to implement this idea in order to obtain a proof of a general theorem similar to Theorem \ref{theorem 2}, without the additional hypothesis that the diffusion coefficient $\sigma $ is an odd function on $\mathbb{R}^{\ast }$.
\vspace{1.2cm}

{\bf Acknowledgements.} This work was supported by a grant of the Romanian National Authority for Scientific Research, CNCS -- UEFISCDI, project number PNII-ID-PCCE-2011-2-0015, and in part by the contract SOP HRD/89/1.5/S/62988.

Last but not least, I dedicate this paper to my dear children, {\newline\centerline{Nicolae, Ana, and \c{S}tefan.}}

\appendix
\section{Construction of skew Brownian motion from the excursion process of a reflecting Brownian motion and an i.i.d. sign choice}\label{Appendix}

In \cite{Ito}, K. It\^{o} showed how to construct the excursion point process of a standard Markov process at a recurrent point of the state space, and he derived the properties of the characteristic measure of the excursion process (Theorem 6.4 in \cite{Ito}). He also suggested (without a proof) that converse is also true: starting with a Poisson point process with characteristic measure satisfying these conditions, one can construct a process whose excursion point process coincides with the initial Poisson point process. Salisbury (\cite{Salisbury1}) showed that in fact It\^{o}'s conditions are not sufficient for the converse, and he found stronger conditions under which the converse is true. His results apply to right-continuous strong Markov processes (for right processes and Ray processes, see \cite{Salisbury2}), and in particular they apply to skew Brownian motion with parameter $\alpha\in [0,1]$ (the cases $\alpha=\frac12$ and $\alpha=1$ correspond to standard Brownian motion, respectively reflecting Brownian motion).

We will use Salisbury's result (Theorem 2 in \cite{Salisbury1}) to give a construction of skew Brownian motion starting from a given reflecting Brownian motion and an i.i.d. sign choice in the sense of Definition \ref{iid sign choice}.

We start by briefly reviewing some elements of It\^{o}'s description of Brownian motion in terms of the Poisson point process of excursions (for more details see \cite{Ito}, or Chapter XII in \cite{Revuz and Yor}).

Consider the standard $1$-dimensional Brownian motion $(W,\mathbb{F}^W)$ on the canonical Wiener space $(C([0,\infty)),\mathcal{B}(C([0,\infty))),P)$. Define the space of excursions by $$U=\{\omega\in C([0,\infty)): 0<R(\omega)<\infty \text{ and } \omega(t)=0 \text{ for } t\ge R(\omega)\},$$
where $R(\omega)=\inf\{t>0: \omega(t)=0\}$, and let $\mathcal{U}$ is the $\sigma$-algebra generated by the coordinate mappings of functions in $U$. Consider $\delta\equiv 0 \in C([0,\infty))$, and let $U_\delta=U\cup\delta$ and $\mathcal{U}_\delta=\sigma (\mathcal{U},{\delta})$.

We define the  \emph{excursion} $e^W_t\in U$ of $W$ at time $t$ by
\begin{equation}\label{construction of excursions}
e_t^W(\omega)(s) = \left\{
\begin{tabular}{ll}
$1_{\{0\le s \le \tau_t(\omega)-\tau_{t-}(\omega)\}} W_{\tau_{t-}(\omega)+s}(\omega),$ & if $\tau_t(\omega)-\tau_{t-}(\omega)>0$ \\
$\delta(\omega)(s)\equiv 0,$ & if $\tau_t(\omega)-\tau_{t-}(\omega)=0$%
\end{tabular}%
\right.,
\end{equation}
where $\tau_t=\inf\{s>0: L_s^0(W)>t\}$, $\tau_{t-}=\inf\{s>0: L_s^0(W)\ge t\}$, and $L_t^0(W)$ is the local time of $W$ at the origin. The \emph{excursion process} of $W$ is the process $e^W=(e^W_t)_{t>0}$ defined on $(C([0,\infty)),\mathcal{B}(C([0,\infty))))$ with values in $(U_\delta,\mathcal{U}_\delta)$.

It is known that $e^W$ is a $(\mathcal{F}^W_{\tau_t})_{t>0}$-Poisson point process, and its characteristic measure $n$ satisfies $n(R>x)=\sqrt{\frac{2}{\pi x}}$, $x>0$. This, together with the law of $u(t)$, $t<R$, conditionally on the value taken by $R$ (see \cite{Revuz and Yor}, Section XII.4) characterizes $n$ uniquely. Denoting by $n_+(\Gamma)=n(\Gamma\cap \mathcal{U}_\delta^+)$ and $n_-(\Gamma)=n(\Gamma\cap \mathcal{U}_\delta^+)$ the restrictions of $n$ to the $\sigma$-algebras $\mathcal{U}_\delta^+$, $\mathcal{U}_\delta^-$, generated by the coordinate mappings of positive excursions ($U_\delta^+$), respectively negative excursions ($U_\delta^-$), we also have $n_+(R>x)=n_-(R>x)=\frac12 n(R>x)$.

Starting from the Brownian motion $W$, we constructed the excursion process $e^W$ (formula (\ref{construction of excursions})). Reversing the construction, that is starting from $e^W$, we may recover $W$ as follows
\begin{equation}\label{construction of process from excursions}
W_t(\omega)=\sum_{s\le L_t(\omega)} e_s(t-\tau_{s-}(\omega)),\qquad t\ge 0,
\end{equation}
where $\tau_t(\omega)=\sum_{s\le t} R(e_s(\omega))$, $\tau_{t-}(\omega)=\sum_{s<t}R(e_s(\omega)$, and $L_t(\omega)=\inf\{s>0: \tau_s(\omega)>t\}$. It is not immediate that in its natural filtration (or the augmented natural filtration), the reconstructed process is a strong Markov process. However, by the results in \cite{Salisbury1}, there exists an enlargement of the initial natural filtration $\mathbb{F}^W$ of $W$ under which the reconstructed process is a strong Markov process.

Define now, for $\alpha\in[0,1]$ arbitrarily fixed, $n_\alpha=\alpha n_+ + (1-\alpha)n_-$, and consider the Poisson point process $e^W$ under this new characteristic measure. As shown in \cite{Salisbury2} (Example 5.7), $n_\alpha$ satisfies the required hypotheses of Theorem $4.1$ (or Theorem $4.2$ in \cite{Salisbury1}), hence there exists a strong Markov process $(X^\alpha,\mathbb{G},P,P^b)$ for which the excursion process $\vert X^\alpha\vert$ is the same as the excursion process of $\vert W\vert $ (the excursions of $X^\alpha$ and $W$ agree modulo their signs), and which is a recurrent extension of $P_0^b$ (the law of Brownian motion starting at $b$ and absorbed at the origin). Also, the process $X^\alpha$ can be identified with the \emph{skew Brownian motion with parameter $\alpha$}, and $P^b$ is the law of $X^\alpha$ starting at $b$.

The above can be taken as a definition of skew Brownian motion with parameter $\alpha$, but there are other, equivalent definitions: in terms of Brownian excursions with randomized signs (\cite{Ito-McKean}), as a solution of stochastic equation or as a limit of random walks (\cite{Harrison-Shepp}), in terms of the speed and scale measure (\cite{Walsh}), etc. It\^{o}-McKean's construction of skew Brownian motion (\cite{Ito-McKean}, pp. 75 -- 78 and p. 115) is perhaps the most intuitive: if we change the sign of each excursions of a reflecting Brownian motion independently with probability $1-\alpha$, the resulting process is a skew Brownian motion with parameter $\alpha\in[0,1]$. In order to make this construction precise we have to label the countable set of Brownian excursions (or equivalent, the set of the corresponding excursion intervals) in a measurable way which depends only on the set of times when the Brownian motion is at the origin, and there are several possibilities to do this. For example, It\^{o} and McKean (\cite{Ito-McKean}) consider  $I_n$ ($n= 1,2,\ldots$) to be the excursion interval (open connected component of $[0,\infty]-\{t\ge 0:W_t=0\}$) which contains the first number in the sequence $1,\frac12,\frac32,2,\frac14,\frac34,\frac54,\frac74,\frac94, \frac{11}4,3,\ldots$, which is not included in $\{t\ge 0:W_t=0\}$ or in one of the intervals $I_1,\ldots,I_{n-1}$ already labeled. Blumenthal (\cite{Blumenthal}, pp. 113 -- 118) considers $I_{n,k}$ ($n,k= 1,2,\ldots$) to be the $k^\text{th}$ leftmost interval with length in $(\frac{1}{n},\frac1{n-1}]$ (when $n=1$, we use the convention $\frac10=\infty$). On pp. \pageref{sign choice} -- \pageref{iid sign choice} we considered is $I_{n,k}$ ($n,k= 1,2,\ldots$) to be the $k^\text{th}$ (leftmost) longest excursion interval of $W$ contained in $(\xi_{n-1},\xi_n)$ and different from $I_{n,1},\ldots,I_{n,k-1}$, where $\xi_0=0$ and $\xi_{n+1}=\inf\{t>\xi_n+1: W_t=0\}$. As pointed out in \cite{Blumenthal}, the actual numbering of the excursions intervals is not important, as long as it depends only on the zero set $Z^W$ of $W$ in a measurable way (to be precise, this means that the random variables representing the left and right endpoints of the numbered interval excursions are $\mathcal{Z}^W$- measurable random variables, where $\mathcal{Z}^W=\sigma(l=l(\omega),r=r(\omega): (l,r) \text{ is an interval excursion of } W)$ is the $\sigma$-algebra generated by the left and right endpoints of excursion intervals of $W$). We will return to this shortly, but for now we will follow It\^{o}-McKean's choice $(I_n)_{n\ge 1}$ of labeling the excursion intervals of $W$.

Consider the corresponding numbering $(e_{n})_{n\ge 1}$ of the excursions of $W$ away from the origin, defined by
\begin{equation}
e_{n}(t)=1_{\{0\le t \le d_n-g_n\}} W_{t+g_{n}}, \qquad t\ge 0, n=1,2,\ldots
\end{equation}
where $I_{n}=(g_{n},d_{n})$, and consider the i.i.d. sequence of random variables $(V_{n})_{n\ge 1}$ with $P(V_{n}=1)=1-P(V_{n})=\alpha\in[0,1]$, which is also assumed to be independent of $W$.

If $e^W=(e^W_t)_{t>0}$ is the $(\mathcal{F}_{\tau_t})_{t>0}$-Poisson point process of excursions of $W$, and we define $\tilde{e}=(\tilde{e}_t)_{t>0}$ by $\tilde{e}_t=V_{n}\vert e_n\vert$ if $e_t^W=e_{n}$, and $\tilde{e}_t=\delta$ otherwise, then $\tilde{e}$ is a $(\tilde{\mathcal{F}}_{ \tau_t})_{t>0}$-Poisson point process, where $\tilde{\mathcal{F}}_t=\mathcal{F}_t \vee \sigma\{V_{n} : d_n<t \}$, and the process $X^\alpha$ reconstructed  from it as in (\ref{construction of process from excursions}) is skew Brownian motion with parameter $\alpha$ (see \cite{Ito-McKean}). Also, by Theorem 4.2 in \cite{Salisbury1} (see also Example 5.7 in \cite{Salisbury2}), there exists a right-continuous complete filtration $\mathbb G$ for which $(X^\alpha,\mathbb G, P, P^b)$ is a strong Markov process.

Note that the above It\^{o}-McKean construction of skew Brownian motion with parameter $\alpha\in[0,1]$ uses in fact only a reflecting Brownian motion $Y=\vert W\vert$, the numbering $(I_n)_{n\ge 1}$ of the interval excursions of $W$ described above (which is the same as the numbering of the excursions intervals of $Y=\vert W\vert$), and an i.i.d sequence of random variables $(V_n)_{n\ge 1}$ with  $P(V_{n}=1)=1-P(V_{n})=\alpha\in[0,1]$, which is also independent of $W$. We conclude with the following lemma, which shows that the numbering of the excursion intervals of Brownian motion $W$ (or equivalent of $Y=\vert B\vert$) is unimportant, so in the above construction of the skew Brownian motion we can consider the numbering of excursions introduced on pp. \pageref{sign choice} -- \pageref{iid sign choice}, that is we can construct the skew Brownian motion starting from a reflecting Brownian motion $Y$ and an i.i.d. sign choice for $Y$ in the sense of Definition \ref{iid sign choice}

\begin{lemma}\label{equivalence of excursion numberings}
If $U$ is an i.i.d. sign choice for a reflecting Brownian motion $Y$, which takes the value $1$ with probability $\alpha\in[0,1]$, and $(I_{i,j})_{i,j\ge 1}$, $(U_{i,j})_{i,j\ge 1}$ are the corresponding interval excursions and their signs in Definition \ref{iid sign choice}, there exists a Brownian motion $W$ with $\vert W\vert =Y$, for which the sequence $(V_n)_{n\ge 1}$ defined by $V_n=U_{i,j}$ if $I_n=I_{i,j}$ ($(I_n)_{n\ge 1}$ is It\^{o}-McKean's numbering of the interval excursions of $W$, as introduced above) is an i.i.d. sequence of random variables with $P(V_n=1)=1-P(V_n=-1)=\alpha$, $n=1,2\ldots$, also independent of $W$.

Conversely, if $(I_n)_{n\ge 1}$ is It\^{o}-McKean's numbering of the excursion intervals of the Brownian motion $W$ as above, and if $(V_n)_{n\ge 1}$ is an i.i.d. sequence of random variables with $P(V_n=1)=1-P(V_n=-1)=\alpha$, $n=1,2\ldots$, also independent of $W$, then the process $U$ defined by $U_t=V_n$ if $t\in V_n$ and $U_t=1$ otherwise, is an i.i.d. sign choice for the reflecting Brownian motion $Y=\vert W\vert$, which takes the value $1$ with probability $\alpha$, in the sense of Definition \ref{iid sign choice}.
\end{lemma}
\begin{proof}
Consider $W$ to be the Brownian motion constructed using It\^{o}-McKean's construction above from the excursion process of $Y$ and i.i.d. sequence of signs $\big(\widetilde{V}_n\big)_{n\ge 1}$ with $P(\widetilde{V}_n=1)=1-P(\widetilde{V}_n=-1)=\frac12$, also independent of $Y$ and $(U_{i,j})_{i,j\ge 1}$. In particular $Y=\vert W\vert$, and $\mathcal{F}^W_\infty\subset \mathcal{F}^Y_\infty\vee \sigma(\widetilde{V}_n:n\ge 1)$.

Since $I_n$ and $I_{i,j}$ are $\mathcal{Z}^W=\mathcal{Z}^Y$-measurable random variables for any $n,i,j\ge 1$, $\sigma(\widetilde{V}_n:n\ge 1)$ and $\mathcal{F}^Y_\infty \vee \sigma(U_{i,j}:i,j\ge 1)$ are independent, and $\sigma(U_{i,j}:i,j\ge 1)$ and $\mathcal{F}^Y_\infty$ are also independent, we obtain
\begin{eqnarray*}
E(V_n \vert \mathcal{F}^W_\infty)&=&\sum_{i,j\ge 1} E\left(U_{i,j} 1_{\{I_{i,j}=I_{n}\}}\big\vert \mathcal{F}^W_\infty\right)\\
&=&\sum_{i,j\ge 1} 1_{\{I_{i,j}=I_{n}\}} E\left(E\left(U_{i,j} \big\vert \mathbb{F}^Y_\infty \cup \sigma(\widetilde{V}_n:n\ge 1)\right) \Big\vert \mathbb{F}^W_\infty\right)\\
&=&\sum_{i,j\ge 1} 1_{\{I_{i,j}=I_{n}\}} E\left(E\big(U_{i,j} \vert \mathbb{F}^Y_\infty \big) \big\vert \mathbb{F}^W_\infty\right)\\
&=&\sum_{i,j\ge 1} 1_{\{I_{i,j}=I_{n}\}} E\left(E\big(U_{i,j}\big) \big\vert \mathbb{F}^W_\infty\right)\\
&=&\sum_{i,j\ge 1} 1_{\{I_{i,j}=I_{n}\}} E\big(U_{i,j}\big)\\
&=&(2\alpha-1)\sum_{i,j\ge 1} 1_{\{I_{i,j}=I_{n}\}} \\
&=&2\alpha -1,
\end{eqnarray*}
which shows that ${V}_n$ is independent of $W$ and that it has the correct distribution.

Similarly, for any $k\ge 1$, $1\le n_1<\ldots <n_k$ and  $u_1,\ldots,u_k\in\{-1,1\}$ we have
\begin{eqnarray*}
&&P\left(V_{n_1}=u_1,\ldots,V_{n_k}=u_k\vert \mathcal{F}^W_\infty\right)=E\left( 1_{\{V_{n_1}=u_1,\ldots,V_{n_k}=u_k\}} \big\vert \mathcal{F}^W_\infty\right)\\
&=&\sum_{\substack{\text{distinct}\\(i_1,j_1),\ldots,(i_xxk,j_k)}} 1_{\{I_{i_1,j_1}=I_{n_1},\ldots, I_{i_k,j_k}=I_{n_k}\}}\cdot\\
&&\phantom{\sum_{\substack{\text{distinct}\\(i_1,j_1),\ldots,(i_k,j_k)}}x}\cdot E\left(E\left(1_{\{U_{i_1,j_1}=u_1,\ldots,U_{i_k,j_k}=u_k\}} \big\vert \mathbb{F}^Y_\infty \cup \sigma(\widetilde{V}_n:n\ge 1)\right) \Big\vert\mathbb{F}^W_\infty\right)\\
&=&\sum_{\substack{\text{distinct}\\(i_1,j_1),\ldots,(i_k,j_k)}} 1_{\{I_{i_1,j_1}=I_{n_1},\ldots, I_{i_k,j_k}=I_{n_k}\}} E\left(E\big(1_{\{U_{i_1,j_1}=u_1,\ldots,U_{i_k,j_k}=u_k\}} \big\vert \mathbb{F}^Y_\infty\big) \Big\vert\mathbb{F}^W_\infty\right)\\
&=&\sum_{\substack{\text{distinct}\\(i_1,j_1),\ldots,(i_k,j_k)}} 1_{\{I_{i_1,j_1}=I_{n_1},\ldots, I_{i_k,j_k}=I_{n_k}\}} E\left(E\big(1_{\{U_{i_1,j_1}=u_1,\ldots,U_{i_k,j_k}=u_k\}} \big) \Big\vert\mathbb{F}^W_\infty\right)\\
&=&\sum_{\substack{\text{distinct}\\(i_1,j_1),\ldots,(i_k,j_k)}} 1_{\{I_{i_1,j_1}=I_{n_1},\ldots, I_{i_k,j_k}=I_{n_k}\}} E\big(1_{\{U_{i_1,j_1}=u_1,\ldots,U_{i_k,j_k}=u_k\}} \big)\\
&=&P\big(U_{i_1,j_1}=u_1,\ldots,U_{i_k,j_k}=u_k \big) \sum_{\substack{\text{distinct}\\(i_1,j_1),\ldots,(i_k,j_k)}} 1_{\{I_{i_1,j_1}=I_{n_1},\ldots, I_{i_k,j_k}=I_{n_k}\}} \\
&=&P\big(U_{i_1,j_1}=u_1\big)\cdot\ldots\cdot P\big(U_{i_k,j_k}=u_k \big)\\
&=&P\big(V_{n_1}=u_1\big)\cdot\ldots\cdot P\big(V_{n_k}=u_k \big),
\end{eqnarray*}
and therefore $(V_n)_{n\ge1}$ is an independent sequence of random variables, also independent of $W$.

Conversely, if $(I_{i,j})_{i,j\ge1}$ and $(U_{i,j})_{i,j\ge 1}$ are the interval excursions of $Y$ (same as of $W$) and the restriction of $U$ to these intervals as in Definition \ref{iid sign choice}, in order to prove the claim we have to show $(U_{i,j})_{i,j\ge 1}$ is an i.i.d. sequence of random variables with the prescribed distribution, and it is also independent of $\mathcal{F}^Y$. The proof being similar with the above, we omit it.
\end{proof}

Summarizing the previous discussion we conclude that the skew Brownian motion with parameter $\alpha\in[0,1]$ can be constructed as a strong Markov process $(X^\alpha, \mathbb{G},P,P^b)$ from a reflecting Brownian motion $Y$ on $[0,\infty)$ and an i.i.d. sign choice for $Y$ which takes the value $1$ with probability $\alpha$, in the sense of Definition \ref{iid sign choice}. Moreover, from the above construction it can be seen that $X^\alpha_t=U_tY_t$ for all $t\ge 0$.
\end{document}